\documentclass[12pt, oneside,reqno]{article}
\usepackage{graphicx}
\usepackage{amsfonts }
\usepackage{amsmath}
\usepackage{mathabx}
\usepackage{fullpage}
\usepackage{amssymb}
\usepackage{amsthm}
\usepackage{tikz}
\usepackage{mathtools}
\usepackage{lmodern}
\usepackage{paralist}
\usepackage[nodisplayskipstretch]{setspace}
\newtheorem{thm}{Theorem}[section]
\newtheorem{conj}[thm]{Conjecture}
\newtheorem{lemma}[thm]{Lemma}

\newtheorem{defn}[thm]{Definition}
\newtheorem{cor}[thm]{Corollary}
\newtheorem{ques}[thm]{Question}

\theoremstyle{definition}
\newtheorem{rem}[thm]{Remark}
\newtheorem{ex}[thm]{Example}

\newcommand{\B}{{\cal B}}
\newcommand{\C}{{\cal C}}
\newcommand{\RR}{{\cal R}}

\newcommand{\rk}{\text{rk}}

\newcommand{\link}{\text{Link}}

\newcommand{\sk}{\text{Skel}}
\newcommand{\J}{{\cal J}}
\newcommand{\nbc}{\text{nbc}}
\newcommand{\Gale}{\text{Gale}}
\newcommand{\Int}{\text{Int}}
\newcommand{\A}{{\cal A}}

\begin{document}
\title{Quasi-matroidal classes of ordered simplicial complexes}
\author {Jos\'e Alejandro Samper\thanks{J.~A.~Samper thanks Isabella Novik for the research assistant positions funded through NSF Grant DMS-1361423.}\\
\small Department of Mathematics\\[-0.8ex]
\small University of Washington\\[-0.8ex]
\small Seattle, WA 98195-4350, USA\\[-0.8ex]
\small \texttt{samper@math.washington.edu}
}
\date{\today}

\maketitle

\begin{abstract}
We introduce the notion of a quasi-matroidal class of ordered simplicial complexes: an approximation to the idea of a matroid cryptomorphism in the landscape of ordered simplicial complexes. A quasi-matroidal class contains pure shifted simplicial complexes and ordered matroid independence complexes. The essential property is that if a fixed simplicial complex belongs to this class for \emph{every} ordering of its vertex set, then it is a matroid independence complex. Some examples of such classes appear implicitly in the matroid theory literature. We introduce various such classes that highlight different apsects of matroid theory and its similarities with the theory of shifted simplicial complexes. For example, we lift the study of objects like the Tutte polynomial and nbc complexes to a quasi-matroidal class that allows us to define such objects for shifted complexes. Furthermore, some of the quasi-matroidal classes are amenable to inductive techniques that can't be applied directly in the context of matroid theory. As an example, we provide a suitable setting to reformulate and extend conjecture of Stanley about $h$-vectors of matroids which is expected to be tractable with techniques that are out of reach for matroids alone. This new conjecture holds for pure shifted simplicial complexes and matroids of rank up to 4. 
\end{abstract}
\section{Introduction}

The term cryptomorphism is an informal mathematical notion that was invented by Birkhoff \cite{MR0227053} in order to capture the phenomenon that a class of objects can be described in several different ways that are not trivially equivalent. Matroids, as an abstract apparatus to study the notion of independence in mathematics, can be defined by a wide variety of axioms that are equivalent, yet have various distinct flavors. Classical matroid cryptomorphisms include, among others, the independence, circuit, basis exchange, submodularity, flat exchange, and closure axioms. Each such set of axioms provides a natural way to study matroids. Furthermore, there are many theorems in matroid theory that seem to be deeply connected to specific axioms: they are quite easy to prove from one point of view and quite hard from another one. For an introduction to the theory of matroids and many existing cryptomorphisms the reader is referred to the books of Oxley \cite{MR1207587}, Welsh \cite{MR0427112} and the book chapters by Bj\"orner \cite{MR1165544} and Ardila \cite{MR3409342}.

Various other cryptomorphisms of matroids have appeared over the years and have turned out to also be useful for many other purposes. Interesting examples come from the theory of simplicial complexes via purity of induced subcomplexes, commutative algebra via the Cohen-Macualayness of the Stanley-Reisner ring of the independence complex and all of its induced subcomplexes (see \cite{MR1453579}),  the theory of polytopes via the matroid basis polytope (see \cite{MR877789}) and optimization via the greedy algorithms working for varying weights (see \cite{MR1165544}). 

Many theorems about matroids appear to have an axiom or a natural set of axioms attached to them in the sense that those axioms play the key role in proving the desired property. For example, the fact that matroids are shellable follows naturally from the exchange axiom, and the theory of internal activities follows from the shellability property. On the other hand, the pure subcomplexes cryptomorphism seems to be a natural consequence of the independence axiom. Also, the behavior of nbc complexes and external activity theories appear to be governed by the circuit axiom. Following this heuristic line of thought, the behavior of the Tutte polynomial would have to be captured by the exchange axiom and the circuit axiom, as it has a natural interpretation in terms of internal and external activities. 

Two particularly interesting cryptomorphisms come from the theory of ordered matroids. In particular, Bj\"orner \cite{MR1165544} proved that a simplicial complex is the independence complex of a matroid if and only if, for every ordering of the vertex set, the induced lexicographic order on the facets is a shelling order. Another outstanding characterization, due to Gale \cite{MR0227039}, is the minimality property in the coordinatewise order, now called the Gale ordering. A family of $d$-element subsets of a fixed set $E$ is the set of bases of a matroid if and only if for every order of $E$ the minimal lexicographic facet is componentwise minimal, that is, if $b_1< \dots < b_r$ are the elements of the smallest basis in the lexicographic order and $b_1'< \dots < b_r'$ are the elements of any other facet, then $b_i \le b_i'$ for all $i$. 

The reason for the last two characterizations to be of a particular interest is the following: they both give a property of ordered simplicial complexes that has to be satisfied for \emph{all} possible orderings of the groundset. It is standard in matroid theory, just as in linear algebra when one has a collection of vectors, to endow the groundset of the independence complex with a total order. For instance, the widely studied nbc complex (see for example \cite{MR468931}) of a matroid is an object that can only be constructed once an order for the groundset of the matroid is fixed. In fact, different orders of the groundset may give many non-isomorphic nbc complexes. Another example comes from the theory of the Tutte polynomial (see for \cite{MR0061366, MR0262095}), a bivariate polynomial with integer coefficients that can be associated to every matroid. The Tutte polynomial encodes all invariants of matroids that satisfy a linear deletion-contraction recurrence. It is known that the coefficients of the Tutte polynomial are non-negative integers, but a combinatorial interpretation of the coefficients of the polynomial is only known once an order for the groundset is fixed. 

Additional motivation to study orderings of the groundset more carefully comes from the theory of shifted complexes. They form a remarkable class of simplicial complexes that became popular because of their simple, yet elegant and useful structure. Shifted complexes appear in the proof of the Kruskal-Katona theorem on face enumeration of simplicial complexes (see \cite{MR0154827,MR0290982}) and in Kalai's algebraic shifting theory \cite{MR1890098} which does the same enumeration while keeping track of more refined invariants of the original simplicial complex. The definition of shifted simplicial complexes, i.e., ordered complexes in which big vertices can be replaced by small vertices without leaving the complex, appears to have the same flavor to that of matroid theory via the exchange axiom. However, the two classes of complexes are quite different: the former relies on a specific order of the groundset and contains many complexes that are not matroids, while most matroids on a fixed groundset are not shifted for any choice of ordering. 

The similarities between the two classes are, however, quite striking. For example, assuming purity in the shifted class, both classes admit quite natural shelling orders (once matroids are ordered) and the combinatorial invariants read from both shelling orders behave quite similarly. Furthermore, both classes admit a very flexible theory of restrictions and contractions, both are closed under a certain type of duality, and in both cases the corresponding Gale orderings have a minimum. In addition, the intersection of both classes of complexes is remarkable: ordered complexes that are simultaneously shifted and matroid independence complexes are sometimes called Schubert matroids; they correspond to the matroids associated to generic points in Schubert strata of (framed) Grassmannian manifolds. 

An even more remarkable and mysterious similarity comes from the theory of combinatorial Laplacians as introduced by Eckmann \cite{MR0013318} and Friedman \cite{MR1622290}. For a simplicial complex $\Delta$, let $(C_\bullet(\Delta), \partial)$ be the simplicial chain complex of $\Delta$ over $\mathbb{R}$ and let $(C_\bullet(\Delta), \delta)$ be the dual complex obtained by using the natural face basis in each degree of the chain complex. For every integer $k$, the operator $D_k:=\delta \partial + \partial \delta$, called the \emph{Laplacian} of $\Delta$, is a self-adjoint operator on $C_k(\Delta)$ that has non-negative real eigenvalues. It is then desirable to relate the spectral theory of $D_k$ to the combinatorial structure of $\Delta$ just as in spectral graph theory: graphs can be viewed as one-dimensional simplicial complexes and the classical spectral theory is a special case of this one. 

It was shown in a series of papers (\cite{MR1697094, MR1844702, MR1926878, MR2087892, MR2218401}) that the eigenvalues of the Laplacians of both matroid independence complexes and shifted simplicial complexes are integer numbers. Furthermore, the eigenvalues can be put into a bivariate generating function, called the spectral polynomial, that satisfies a special kind of recurrence similar to the deletion-contraction recurrence for matroids, except that it has an error correction term coming from relative topology. This is a very rare property: the Laplacian operators of most simplicial complexes on a fixed vertex set do not have integral spectra. This leads naturally to the following question that has been repetedly asked. 

\begin{ques}[\cite{Reiner-slides, Duval-slides, MR1926878}]\label{ques:duv-rei} Is there a class of simplicial complexes that contains matroid independence complexes and shifted simplicial complexes, and explains the integral Laplacian phenomenon? 
\end{ques}

Yet another reason for a more detailed study of ordered complexes comes from the theory of $f$-vectors of matroids.  The $f$-vector $(f_0, \dots, f_d)$ of a rank-$d$ simplicial complex $\Delta$ enumerates faces of each rank, i.e the entry $f_i$ counts the number of independent sets (or faces of the independence complex) of rank $i$. It is natural to ask for a characterization of the possible $f$-vectors of matroids. This question has been answered entirely and quite succesfully for other classes of simplicial complexes: for instance the class of all simplicial complexes \cite{MR0154827,MR0290982}, the class of Cohen-Macaulay simplicial complexes \cite{MR0572989}, and the class of simplicial polytopes \cite{MR635368, MR563925}. The $h$-vector of a matroid is an invertible transformation of the $f$-vector that is sometimes more convenient. Thus an equivalent question is that of classifying the possible $h$-vectors of matroids. The advantage here is that the $h$-vector theory of a matroids has a combinatorial realization provided by the lexicographic shelling order of the bases of the matroid, after fixing one ordering of the groundset (see Bj\"orner \cite{MR1165544} for details).

Even though the family of $h$-vectors of matroids is believed to be quite wild and hopeless to fully classify, there are several restrictions the possible values such a vector can take. An astonishing result of Adiprasito, Huh, and Katz \cite{1511.02888}, that builds on previous work of Huh and Katz \cite{MR2904577, MR2983081}, proves that the $f$-vector of the nbc complex of a matroid is log concave, thus resolving a long standing conjecture due to Heron, Rota and Welsh. This imposes strong restrictions on the family of $f$-vectors of matroids. 

One of the most intriguing questions in matroid theory concerns the $h$-vector of the independence complex of a matroid. Given an ordered rank- $d$ matroid, the lexicographic order on the bases is a shelling of the independence complex and the same holds for nbc bases. It follows that the independence complex and the nbc complex of a matroid are Cohen-Macaulay, thus the corresponding $h$-vectors are $O$-sequences. In other words, there is a family of monomials $\cal O$ closed under divisibility with exactly $h_i$ monomials of degree $i$ for every $i=1, \dots, \, d$. The known general constructions for $\cal O$ are not combinatorial and the stucture of $\cal O$ has little to do with the structure of the matroid. It is easy to find several $O$-sequences that are not $h$-vectors of matroids, thus one might ask if there are other conditions that $h$-vectors of matroids have to satisfy. In 1977 Stanley posited the following conjecture on $h$-vectors of matroids. 

\begin{conj}[\cite{MR0572989}]\label{Stanley} The $h$-vector of a matroid independence complex is a {\bf pure} $O$-sequence. 
\end{conj}

Being a pure $O$-sequence simply means that there exists a multicomplex $\cal O$ that realizes the $h$-vector of the independence complex of the matroid with the additional property that all maximal monomials of $\cal O$ with respect to divisibility have the same degree. Stanley's conjecture has received a lot of attention in the last few decades and is known to hold for various special classes of matroids \cite{MR1888777, MR2578897, MR3104512, MR2950481, MR2880645, MR3090174, MR3262195, OhThesis, MR3339029, Dall}. See  \cite{MR3339029} or Dall \cite{Dall} for details about the current status of the conjecture. 

More is known about matroid $h$-vectors. Hibi \cite{MR989204} found a set of inequalities satisfied by pure $O$-sequences and Chari \cite{MR1422892} provided a topological decomposition of the  
independence complex a matroid that implies Hibi's inequalities for the $h$-vector. Furthermore, Swartz \cite{MR1983365} provided an algebraic version of these inequalities in the artinian reduction of the Stanley-Reisner ring of the independence complex. Juhnke-Kubitzke and Van Dinh \cite{1604.02938} proved such inequalities for $h$-vectors of nbc complexes of representable matroids building on work of Huh \cite{MR3286530}.

In \cite{MR3339029} Klee and the author conjectured a more refined version of Stanley's conjecture that predicts the existence of a multicomplex $\cal O$ whose combinatorial structure is related to the combinatorial structure of the underlying matroid. The idea is to use the shelling order: each monomial of $\cal O$ corresponds to a basis of the matroid and depends on the restriction set of the basis. There are two main obstructions to such an approach. The first one is that constructing a matroid by using the shelling order yields intermediate complexes that do not come from matroids. The second one is that the purity cannot be expected to hold during the whole process, which means that a substitution for purity is required in the inductive setting. 

The main goal of this paper is to discuss three quasi-matroidal classes of complexes, i.e., classes of ordered simplicial complexes that contain all ordered matroids and all pure shifted complexes and such that a fixed simplicial complex that belongs to the quasi-matroidal class in question for every order of its vertex set is necessarily the independence complex of a matroid. Examples of quasi-matroidal classes are implicitly known in the literature. For instance, the class of ordered complexes with the property that the lexicographic order of the facets is a shelling order is an example. The class of ordered complexes for which the Gale ordering has a minimum is another example. Various new quasi-matroidal classes will be described in this paper. Three of these classes are deeply related to three classical cryptomorphisms: the independence, exchange, and circuit axioms. The three classes are pairwise different, enjoy some interesting properties of matroids and elucidate similarities between matroids and shifted complexes.  

The following list summarizes our results on quasi-matroidal classes: 
\begin{compactitem}
\item Each class carries a piece of matroid theory, and thus effectively provides a way to classify some matroid properties according to the classical matroid properties that need to be extended to achieve analogous results. 
\item The independence quasi-matroidal class implies that many of induced subcomplexes are pure and provides a formal dependence relation between the independence axiom and the purity of induced subcomplexes.

\item The exchange quasi-matroidal class turns out to be a subclass of the complexes that are shellable in lexicographic order. It provides a meaningful internal activity theory and is closed under what we call Gale truncations, which makes it a suitable class of complexes to do induction on the number of facets. 

\item The circuit quasi-matroidal class gives a good theory of fundamental circuits and a well behaved nbc complex theory.

\item The intesection of the exchange and circuit quasi-matroidal classes imply that the nbc complex is pure shellable and that there is a well behaved Tutte polynomial whose coefficients can be interpreted combinatorially in terms of internal and external activities.  

\item The complexes belonging to the independence and exchange quasi-matroidal classes that also satisfy another technical restriction admit a reformulation of the conjecture of Klee and the author \cite{MR3339029}. This conjecture turns the purity part of Stanley's conjecture into a poset theoretic restriction, which is suitable for induction on the number of bases (or facets). 

\item The new conjecture is satisfied by Gale truncations of matroids of rank up to four and Gale truncations of the internally perfect matroids defined by Dall \cite{Dall}. Furthermore, in order to verify the conjecture for complexes of rank $d$, it suffices to verify it for complexes with no more than $2d-1$ vertices.

\item It is also shown that the new conjecture is satisfied by pure shifted complexes. This provides a new proof of Stanley's conjecture for Schubert matroids. 
\end{compactitem}

The paper is organized as follows. Section 2 provides some background and definitions. Section 3 introduces the notion of a quasi-matroidal classes and studies some basic properties. Section 4 is devoted to the independence, exchange, and circuit quasi-matroidal classes and discusses some basic properties of each resulting class of complexes. Section 5 deals with the first facet property, a condition on the local structure of the ordered simplicial complexes that provides a lot of flexibility to play with combinatorial operations that preserve some quasi-matroidal classes. Section 6 discusses complexes in the intersection of the exchange and circuit quasi-matroidal classes and develops a theory of Tutte polynomials and nbc complexes. Section 7 gives the relaxation of the conjecture of Klee and the author for complexes with the first facet property that belong to the independence and exchange quasi-matroidal classes. It ends with a proof of the new conjecture for shifted complexes. Section 8 contains open problems, brief descriptions of future research projects and various comments about connections to the existing literature. 

\noindent{\bf Acknowledgements}
I would like to thank Isabella Novik, Steve Klee, Vic Reiner, Ernest Chong, Jeremy Martin and Gillaume Chapuy for many interesting conversations and suggestions.   

\section{Preliminaries}\label{sct:prelim}
An \emph{ ordered simplicial complex} is a pair $\Psi = (E, \Delta)$ where $E$ is a totally ordered finite set and $\Delta\subseteq 2^E$ is a simplicial complex, that is, if $A \in \Delta$ and $B\subset A$, then $B\in \Delta$. Matroid terminology will be used throughout the paper. Elements of $\Delta$ are called \emph{independent sets}. Maximal under inclusion independent sets are called \emph{bases}. The set of bases is denoted by $\B$. A complex $\Psi$ is called \emph{pure} if all the bases have the same cardinality. The smallest lexicographic basis is denoted by $B_0$. Minimal elements not in $\Delta$ are called \emph{circuits}. The set of circuits is denoted by $\C$. The $\emph{rank}$ of an independent set is equal to its cardinality and the \emph{rank} of a subset $A$ of $E$ is the maximum rank of an independent set contained in $A$. Abusing notation, define the rank of $\Psi$ to be the rank of $E$. The rank of $\Psi$ is usually denoted by $d$.  Two ordered complexes $\Psi=(E,\Delta)$ and $\Psi'=(E', \Delta')$ are said to be isomorphic if $|E| = |E'|$ and the unique ordered bijection of $E$ and $E'$ induces a bijection between $\Delta$ and $\Delta'$. 

 A \emph{loop} of $\Psi$ is an element of $E$ that is not in any basis. A \emph{vertex} is an element of $E$ that is not a loop. The set of vertices of $\Psi$ is denoted by $V(\Psi)$. A \emph{coloop} is an element of $\Psi$ that belongs to every basis. For $A\subset E$, define the \emph{restriction} $\Psi|_A$ to be the pair $(A, \Delta|_A)$, where $\Delta_A = \{ I\in \Delta \, : \, I \subseteq A\}$. The \emph{deletion} $\Psi\slash\{e\}$ of an element $e$ that is not a coloop is the restriction to $E\backslash \{e\}$. The \emph{contraction} $\Psi \backslash\{e\}$ of an element $e$ that is not a loop is the complex $(E\backslash\{e\}, \link_\Delta(e))$, where $\link_{\Delta}(e)= \{I \in \Delta \, | \, e\not\in I \text{ and } I\cup \{e\}\in \Delta\}$. The \emph{contraction} $\Psi\backslash I$ of an independent set $I$ is the complex that results from contracting the vertices of $I$ in any order. For an independent set $I$, let $B_{I,0}$ be the smallest lexicographic basis of $\Psi \backslash I$. The complex $(\{e\}, \{\emptyset\})$ is denoted by $\Psi_{\text{loop}}$ and the complex $(\{e\}, \{\emptyset, \{e\}\})$ is denoted by $\Psi_{\text{coloop}}$. 
 
Given two ordered complexes $\Psi = (E, \Delta)$ and $\Psi' = (E', \Delta')$, a \emph{shuffle} $s(E, E')$ of $E$ and $E'$ is an ordered set with order preserving inclusions $j:E\to X$ and $j':E'\to X$, such that $j(E)\cap j'(E') = \emptyset$ and $j(E)\cup j(E') = s(E, E')$. Given a shuffle $s(E, E')$ of $E$ and $E'$, the join $\Psi*_{s(E,E')}\Psi':=(s(E,E'), \Delta*\Delta')$ is the complex whose independent sets are of the form $j(I)\cup j'(I')$ for some $I\in \Delta$ and $I'\in \Delta'$. If the ranks of $\Psi$ and $\Psi'$ are equal, the connected sum $\Psi \# \Psi'$ is the complex obtained by identifying $B_0$ and $B_0'$ via the unique order preserving bijection. The rank-$k$ skeleton $\Psi^{(k)}$ of $\Psi$ is the the complex $(E, \Delta^{(k)})$ whose independent sets are the independent sets of $\Delta$ of rank at most $k$. 

Let $B$ be a basis. An element $e\in E\backslash B$ is called \emph{externally active} with respect to $B$ if there is a circuit $C\subset B\cup\{e\}$ such that $e$ is the smallest element of $C$. An element $e\in E\backslash B$ is \emph{externally passive} if it is not externally active. The sets of externally active and passive elements of $B$ are denoted by $EA(B)$ and $EP(B)$ respectively. An element $b \in B$ is called \emph{internally active} if $B$ is the smallest basis in lexicographic order that contains $B\backslash\{b\}$. Equivalently, there is no $b'<b$ that is not in $B$ and such that $(B\backslash\{b\})\cup\{b'\}\in \B$. An element $b\in B$ is \emph{internally passive} if it is not internally active. The sets of internally active and passive elements of $B$ are denoted by $IA(B)$ and $IP(B)$ respectively.

A \emph{broken circuit} of $\Psi$ is a subset $D$ of $E$ that is of the form $C-\{c\}$, where $C$ is a circuit and $c$ is the smallest element of $c$. The \emph{nbc complex} $\nbc(\Psi)$ of $\Psi$ is the complex $(E, \Gamma)$ whose bases are the bases of $\Psi$ that do not contain a broken circuit. 

The \emph{Gale ordering} of a pure ordered complex $\Psi$ is the poset $\Gale(\B, <_{\Gale})$ defined by the following relation: $B <_{\Gale}B'$ if and only if the elements of $B = \{ b_1 < \dots < b_d\}$ and $B'= \{b_1'< b_2'< \dots < b_d'\}$ satisfy   $b_i \le b_i'$ for every $i$. Let $\cal J$ be an order ideal of $\Gale(\B, <_{\Gale})$. The \emph{Gale truncation} at $\cal J$ is the complex $\Psi[{\cal J}] := (E, \Delta[{\cal J}])$ whose bases are the elements of $\cal J$. 

The \emph{Internal poset} $\Int(\Psi) = (\B, <_{\text{int}})$ of $\Psi$ is defined by the relation $B_1\le_{\text{int}} B_2$ whenever $IP(B_1) \subseteq IP(B_2)$. 

An ordered complex $\Psi$ is an ordered matroid if and only if either of the following three equivalent axioms is satisfied:  \begin{compactenum} \item[i.] {\bf Independence axiom:} If $I_1$ and $I_2$ are independent sets such that $|I_1|>|I_2|$ then there is an element $i\in I_1\backslash I_2$ such that $I_2\cup\{i\} \in \Delta$. 
\item[ii.] {\bf Exchange axiom:} If $B_1$ and $B_2$ are bases and $b_1$ is an element in $B_1 \backslash B_2$, then there is $b_2\in B_2\backslash B_1$ such that $B_1\backslash\{b_1\}\cup \{b_2\}$ is a basis. 
\item[iii.] {\bf Circuit axiom:} If $C_1$ and $C_2$ are circuits and $c\in C_1\cap C_2$, then there is a circuit $C_3$ contained in $(C_1 \cup C_2)\backslash c$. 
\end{compactenum}
A simple consequence of matroid duality (e.g see \cite{MR0227039}) is that $\Gale(\Psi)$ has a minimum and a maximum whenever $\Psi$ is an ordered matroid. Furthermore, Gale showed that this is a property that in fact characterizes matroids. 
\begin{thm}\label{thm:shell}\emph{\cite{MR1165544}} A simplicial complex $\Delta$ is the independence complex of a matroid if and only if the the Gale poset of $\Psi = (E, \Delta)$ has a minimum for every ordering $E$ of the vertex set of $\Delta$. \end{thm}

Notice that the order of $E$ is not used at all in the definition of a matroid. The \emph{ordered uniform matroid} of rank $d$ over an ordered set $E$ is the complex $U_{E,d} = (E, X_d)$ whose bases are all the $d$-subsets of $E$. 

An ordered complex $\Psi=(E, \Delta)$ is \emph{shifted} if the following holds: if $B$ is a basis and $i<j\in E$ are such that $i\not\in B$ and $j\in B$, then $B\backslash \{i\} \cup \{j\}$ is also a basis. Equivalently, $\Psi$ is a Gale truncation of $U_{E,d}$ for some $d$. The Gale ordering of a shifted complex is isomoprhic to an ordered ideal of Young's lattice of integer partitions. 

The \emph{$f$-vector} of a rank-$d$ complex $\Psi$ is the vector $(f_{0}, f_0, \dots, f_{d})$ where $f_i$ is the number of independent sets of rank $i$. Notice that the empty set is the only independent set of rank $0$, thus $f_0= 1$. The \emph{$h$-vector} $(h_0, \dots, h_d)$ of $\Psi$ is given by the following polynomial relation \begin{equation}h(\Psi, x) := \sum_{j=0}^d h_jx^j = \sum_{j=0}^d f_{j}t^j(1-t)^{d-j}.\end{equation} The $h$-vector carries the same information as the $f$-vector and is sometimes more convenient, in particular, when studying simplicial complexes through the lens of commutative algebra. 

A \emph{shelling order} of a pure complex $\Psi$ is an order $B_1, \dots B_k$ of the bases such that for every $i<j$ there is $k\le j$ and $b\in B_j$ such that $B_i\cap B_j\subseteq B_k\cap B_j = B_j\backslash\{b\}$. The complex $\Psi$ is said to be \emph{shellable} if it admits a shelling order. The following property holds for every shelling order $B_1, \dots , \, B_k$ of a complex $\Psi$: For every $j=1, \dots k$ there is a unique minimal subset $\RR(B_j)$ of $B_j$ such that for every $i<j$ the set $\RR(B_j)$ is not contained in $B_i$. It turns out that \begin{equation}h(\Psi, x) = \sum_{j=1}^k x^{|\RR(B_j)|}. \end{equation}

It is known that both, ordered and pure shifted complexes are shellable. The lexicographic order of the bases is a shelling order. Again, this is another matroid defining property.

\begin{thm}\label{thm:Gale}\emph{\cite{MR0227039}} A simplicial complex $\Delta$ is the independence complex of a matroid if and only if the the Gale poset of $\Psi = (E, \Delta)$ has a minimum for every ordering $E$ of the vertex set of $\Delta$.
\end{thm}

A pure ordered complex $\Psi = (E, \Delta)$ is \emph{vertex decomposable} if and only if either one of the following holds: 
\begin{compactitem}
\item $\Psi$ has exaclty one basis.
\item There exists a vertex $e$ of $\Psi$ such that $\Psi\slash \{e\}$ is vertex decomposable of the same rank of $\Psi$ and $\Psi \backslash  \{e\}$ is vertex decomposable. 
\end{compactitem}

If $\Psi$ is a vertex decomposable with more than one basis, a vertex $e$ of $\Psi$ that satisfies the second condition of vertex decomposability is called a \emph{shedding vertex}. It is a theorem of Billera and Provan \cite{MR593648} that every non coloop vertex decomposable complex is shellable. They furthermore showed that matroids are vertex decomposable: any vertex is a shedding vertex. Pure shifted complexes are also vertex decomposable: the largest vertex is always a shedding vertex. 

For two sets $A, B$, let $A\triangle B$ be their symmetric difference, i.e, the set $(A\backslash B) \cup (B\backslash A)$. Whenever a subset of a small set is considered we omit parethenses and commas to simplify notation. For example, the subset $\{2,4\}$ of $\{1,2,3,4,5\}$ is denoted by $24$. 

\section{Quasi-matroidal classes of ordered complexes}

In this section $\cal A$ denotes a class of ordered complexes. We say that $\cal A$ is \emph{closed under joins} if for every pair of complexes $\Psi$ and $\Psi'$ in $\cal A$ and every shuffle $s$ of their groundsets, the join $\Psi *_s \Psi'$ is a complex in $\cal A$. We say that $\cal A$ is \emph{closed under deletions} if for every complex $\Psi$, the deletion of the largest element of the groundset yields a complex in $\cal A$. Finally, we say that $\cal A$ is \emph{closed under contractions} if for every complex $\Psi$ the contraction of every independent set of $\Psi$ is a complex in $\cal A$. The following notion encapsulates the central type of objects we will study.

\begin{defn}\label{def:quasicrypt} A class $\cal A$ of ordered simplicial complexes is called a \emph{quasi-matroidal} if the following conditions are satisfied: 
\begin{compactenum}
\item Every ordered matroid is an object in $\cal A$. 
\item If $\Delta$ is a simplicial complex with vertex set $X$ and for every order $E$ of $X$, the pair $(E, \Delta)$ is in $\cal A$, then $\Delta$ is a matroid independence complex. 
\item Every pure shifted complex is in $\cal A$. 
\item $\cal A$ is closed under joins, deletions and contractions. 
\end{compactenum}
\end{defn}

Theorems~\ref{thm:shell} and \ref{thm:Gale} provide two different examples of quasi-matroidal classes.

\begin{ex}\label{ex:shell}  The following two classes are quasi-matroidal: 
\begin{compactenum}
\item The class LEX of all pure ordered complexes closed under joins, deletions and contractions for which the lexicographic order on the bases is a shelling order.
\item The class GALE of all pure ordered complexes closed under joins, deletions and contractions, for which the Gale poset has a unique minimal basis. 
\end{compactenum}
\end{ex}

A slightly bigger class that contains both of the above classes is the following one: 
\begin{ex}
Let PURE denote the class of all pure ordered complexes such that a complex $\Psi = (E, \Delta)$ is in PURE if and only if one of the following is satisfied: 
\begin{compactitem}\item $\Psi$ has exactly one basis or,  
\item The deletion $\Psi\slash \{v\}$ is in PURE and has the same rank as $\Psi$ if $v$ denotes the largest non-coloop vertex of $\Psi$ and every contraction of $\Psi$ is in PURE. 
\end{compactitem}
\end{ex}
Pure shifted complexes as well as matroids are easily seen to be in PURE. In fact, PURE is a quasi-matroidal class due to the following classical theorem. 

\begin{thm}\emph{\cite{MR1453579}} A simplicial complex $\Delta$ is the independence complex of a matroid if and only if every induced subcomplex is pure. 
\end{thm}
 
PURE explains some of the first pleasant similarities between pure shifted complexes and matroids. 
\begin{thm}\label{thm:PURE} Every ordered complex in PURE is vertex decomposable and hence shellable. 
\end{thm}
While the classes in the examples above explain various similarities between shifted complexes and matroid independence complexes, they are too big and contain many complexes that are far from shifted complexes or matroids. For instance, all of them contain the complex $\Psi = ([4], \Delta)$ with bases $12, 13, 24$. This complex is the path with three edges, which is the canonical example of a complex whose Laplacian has non-integral spectra. Thus it is desirable to consider smaller quasi-matroidal classes, so that the complexes belonging to such classes share deeper structural properties with matroids and shifted complexes. 

In order to do so, we introduce a few basic constructions of quasi-matroidal classes. Given two quasi-matroidal classes $\A$ and $\A'$, the class $\A\cap \A'$ of all complexes contained in both $\A$ and $\A'$ is clearly quasi-matroidal. A class $\A'$ is called a subclass of $\A$ if all the elements of $\A'$ are elements of $\A$. 

\section{Three quasi-matroidal classes}\label{sct:axioms}
The purpose of this section is to introduce three new quasi-matroidal classes that will be studied throughout the paper. 

\begin{defn} Let $\Psi=(E, \Delta)$ be an ordered simplicial complex. The classes QI, QE and QC are defined by the following axioms. 
\begin{compactitem}
\item {\bf Quasi-independence Axiom (QI):} $\Psi$ is pure and for every pair of independent sets $I_1, I_2$, if $|I_1|>|I_2|$ and $I_1\backslash I \subseteq B_{I,0}$ for some $I\subseteq I_1\cap I_2$, then there exists $e\in I_1\backslash I_2$ such that $I_2\cup\{e\}$ is independent. 
 \item {\bf Quasi-exchange Axiom (QE):}  $\Psi$ is pure and for every pair $B_1,\, B_2$ of bases of $\Delta$, if $b_1 \in B_1\backslash B_2$ satisfies $b_1 > \max B_2\backslash B_1$, then there is $b_2 \in B_2\backslash B_1$ such that $(B_1\backslash \{b_1\})\cup \{b_2\}\in\Delta$.
 \item {\bf Quasi-circuit Axiom (QC):} If $C_1$, $C_2$ are distinct circuits $(E,\Delta)$ and $c \in C_1\cap C_2$ such that $c<\max C_1\triangle C_2$, then there is a circuit $C_3$ of $(E,\Delta)$ contained in $(C_1\cup C_2)\backslash\{c\}$.
\end{compactitem}
\end{defn}

The first goal is to show that the classes QI, QE, and QC are quasi-matroidal. It is straightforward to see that ordered matroids belong to the three classes. The second quasi-matroidal axiom is also straightforward: removing the conditions of the order in QI, QE, and QC yields the classic independence, exchange and circuit axioms of matroid theory. On the other hand it is an interesting exercise to show that shifted simplicial complexes satisfy the axioms.

\begin{thm} If $\Psi$ is a shifted complex, then $\Psi$ belongs to QI, QE and QC. 
\end{thm}

\begin{proof} To prove that $\Psi$ belongs to QI notice that if $I$ is any independent set disjoint from $B_0$, then $B_{I,0}$ is the initial subset of $B_0$ of size $d-|I|$, where $d$ is the rank of $B_0$. Thus if $I_1$ and $I_2$ are independent sets that satisfy the conditions of QI with a witness $I\subseteq I_1\cap I_2$. Then $I_1\backslash I$ is a subset of the first $d-|I|$ elements of $B_0$ and $B_{I_2, 0}$ consists of the first $d-|I_2|$ elements of $B_0$. Hence $B_{I_2, 0} \subseteq B_{I,0}$. Then $B_{I_2,0}\cap (I_1\backslash I) \not= \emptyset$. Otherwise, $|B_{I_2,0}\cup (I_1\backslash I)| = (d-|I_2|) + (|I_1|- |I|) > d - |I| = |B_{I,0}|$ which is a contradiction since both are subsets of $B_{I,0}$. 

To prove that $\Psi$ belongs to QE, notice that if $B_1$ and $B_2$ are bases and $b_1\in B_1\backslash B_2$ is bigger than any element in $B_2\backslash B_1$, then shiftedness of $\Psi$ allows to choose any $b_2\in (B_2\backslash B_1)$ to replace $b_1$ in $B_1$. 

To prove that $\Psi$ belongs QC, let $C_1$ and $C_2$ be circuits of $\Psi$ and let $c \in C_1\cap C_2$ satisfy the conditions for QC. Assume that $I:=(C_1 \cup C_2)\backslash c$ is independent and let $c' = \max{C_1 \triangle C_2}$. By shiftedness, $I':=(I\backslash\{c'\})\cup\{c\}$ is independent, but it also contains either $C_1$ or $C_2$, which is a contradiction. 
\end{proof}

It is clear that every quasi-matroidal class can be transformed into a matroid cryptomorphism. While matroid cryptomorphisms give rise to the same class, there are various quasi-matroidal classes each of which highlights different aspects of matroid theory. The following theorem shows that the three defined classes are indeed different.

\begin{thm}\label{thm:ex} The classes QI, QE and QC are all distinct, furthermore, no class is contained in another one. 
\end{thm} 

\begin{proof} For every pair of axioms one has to provide examples of complexes that satisfy one but not the other axiom. 
\begin{compactitem} 
\item The complex $\Psi_1=(\{1,2,3,4\}, \Delta_1)$ with bases $12,13,14,34$ satisfies QI, QE, but the pair of circuits $23,24$ contradicts QC. 

\item The complex $\Psi_2 = (\{1,2,3,4\}, \Delta_2)$ with bases $14, 24,23,34$ satisfies QI, but the bases $14$ and $23$ show that it does not satisfy QE, and the circuits $13$, $14$ show that it does not satisfy QC.  

\item The complex $\Psi_{3} = (\{1,2,3,4\}, \Delta_3)$ with bases $12, 13, 23, 34$ satisfies QC, but fails QI and QE. 

\item The complex $\Psi_{4} = (\{1,2,3,4,5\}, \Delta_4)$ with bases $13,14,23,24,25$ satisfies QE, but not QI or QC. 
\end{compactitem}
\end{proof}

Notice that the proof of the theorem shows that QC is not contained in QI$\cap$QE. On the other hand, it is a straightforward exercise in graph theory to check that a rank-two complex that belongs to  QE$\cap$QC also belongs to QI. Nevertheless, the proof fails in rank-three and it is natural to ask the following question.

\begin{ques}Are the classes QI$\cap$QE, QI$\cap$QC, QE$\cap$QC and QI$\cap$QE$\cap$QC all distinct?
\end{ques}

Regardless of the answer, the definitions provide various classes of simplicial complexes and it should come as no surprise that each of them shares some structural properties with the family of ordered matroids. The goal for the rest of this section is to start developing the first steps of a theory for these classes of complexes that includes some analogs of matroid properties as well as to provide various types of constructions that can be performed within a given class.

\subsection{The quasi-independence class}\label{sect:indep}
We now show that QI is a quasi-matroidal class with a little extra structure. 

\begin{thm}\label{thm:QIconst} Let $\Psi= (E, \Delta)$ and $\Psi'=(E', \Delta')$ be ordered complexes in QI. 
\begin{compactenum}
\item[i.] If $s(E,E')$ is a shuffle of $E$ and $E'$ then the join $\Psi*_{s(E, E')} \Psi'$ is in QI.
\item[ii.] If $v\in E\backslash B_0$ then the contraction $\Psi \slash \{v\}$ is in QI. 
\item[iii.] If $v$ is the largest non-coloop vertex of $\Psi$, then $\Psi\backslash \{v\}$ is in QI. 
\item[iv.] If $0\le k \le \rk(\Psi)$ then the skeleton $\sk_k(\Psi)$ is in QI. 
\item [v.] If $\rk(\Psi) = \rk(\Psi')$ then the connected sum $\Psi \#_{\varphi} \Psi'$ is in QI. 
\end{compactenum}
Parts [i.], [ii.], and [iii.] imply that QI is a quasi-matroidal subclass of PURE. 
\end{thm}
\begin{proof}
Part [i.] follows from the fact that joins preserve purity and commute with links, i.e., if $\Delta_1$ and $\Delta_2$ are complexes and $I_1, I_2$ are faces of $\Delta_1$ and $\Delta_2$ then $\link_{\Delta_1*\Delta_2}(I_1\cup I_2) = \link_{\Delta_1}(I_1)* \link_{\Delta_2}(I_2)$. 


For [ii.] notice that if $I_1$ , $I_2$ and $I\subset I_1\cap I_2$ satisfy the hypotheses of the axiom in $\link_\Delta(v)$, then $I_1\cup \{v\}, I_{2}\cup \{v\}$ and $I\cup \{v\}$ satisfy the hypotheses in $\Psi$. We may therefore use the QI axiom in $\Psi$ to obtain the result. 

To prove part [iii.]  notice that $I$ is an independent set in $\Psi$ that does not contain $v$ and then $B_{I,0}$ does not contain $v$. Indeed if $v \in B_{I,0}$, let $B = (I\cup B_{I,0}) \backslash \{v\}$ and use the QI with $B$ and $B_0$, i.e, there is $u\in B_0\backslash B$ such that $B\cup\{u\}$ is independent. The vertex $u$ is not a coloop of $\Psi$ (it does not belong to the basis $B_{I,0}\cup I$) and is therefore smaller than $v$. It follows that $(B\cup\{u\}) \backslash I$ is a basis of $\Psi\slash I$ smaller in lexicographic order than $B_{I,0}$, a contradiction. 

Part [iv.]  follows from the equality $\link_{\sk_k\Delta} I = \sk_{k-|I|}(\link_{\Delta} I)$ and the fact that the smallest lexicographic facet of $\sk_{k}\Psi$ is the smallest lexicographic $k$-face of $B_0$. 

For part [v.] assume that $I_1, I_2$, and $I\subseteq I_1\cap I_2$ satisfy the hypotheses in the connected sum. If $I_1$ and $I_2$ are independent in $\Psi$ and $\Psi'$, respectively, then $I_1\cap I_2$ is a subset of $B_0 = B_0'$ (under the natural identification), and so $B_{I,0} = B_0 \backslash I = B_0'\backslash I$. Hence $I_1 \subseteq B_0 = B_0'$ and we may apply the axiom for $\Psi'$. If both independent sets come from the same $\Psi$ or $\Psi'$, then the QI axiom can be applied as coming from $\Psi$ or $\Psi'$. 

\end{proof}

The following theorem shows that QI refines PURE in a special sense: a larger family of induced subcomplexes are pure for complexes in QI.

\begin{thm} Let $\Psi= (E, \Delta)$ be a complex in QI and let $A\subseteq E$ be a subset such that $\rk(A) = |B_0\cap A|$. Then $\Psi|_A$ is pure. 
\end{thm}
\begin{proof}Notice that if $I$ is an independent set in $\Psi|_A$, it is possible to apply the QI axiom with $B_0\cap A$ and $I$ to extend $I$ to an independent set $B$ of $\Psi|_{A}$. Then $|B| = |B_0\cap A| = \rk(A)$. It follows that $B$ is a basis of $\Psi_{A}$ and this implies purity. 
\end{proof}

\subsection{The Exchange axiom}\label{ssct:exch} 
The theory of shellability of simplicial complexes is best studied in the language of facets and it is natural that it should follow from conditions on the set of bases of a simplicial complex. It turns out that QE is a suitable class to apply this technology. 

\begin{thm}\label{thm:QEconst} Let $\Psi= (E, \Delta)$ and $\Psi'=(E', \Delta')$ be ordered complexes in QE. 
\begin{compactenum}
\item[i.] If $s(E, E')$ is a shuffle of $E$ and $E'$ then the join $\Psi*_{s(E, E')} \Psi'$ is in QE.
\item[ii.] If $U \subset E$ is any subset such that $\Psi|_U$ is pure and of the same rank as $\Psi$, then the restriction $\Psi|_U$ is in QE. 
\item[iii.] If $A\subseteq$ contains $B_0$ and all the elements smaller than the largest non-coloop of $B_0$, then the restriction $\Psi|_A$ is in QE. 
\item[iv.] If $v$ is a vertex of $\Psi$ then the contraction $\Psi \slash v$ is in QE. 
\end{compactenum}
In particular, QE is a quasi-matroidal sublclass of PURE.
\end{thm}
\begin{proof}
Part [i.] follows from the fact that joins preserve purity, and bases in the join correspond to pairs of bases, coming from each complex. Then the QE axiom applies to either basis for each element that can be switched. 

For [ii.] notice that if $B_1$ and $B_2$ are bases of $\Psi|_U$ then applying the QE axiom with $B_1$ and $B_2$ produces bases of $\Psi|_U$. 

For [iii.] it suffices to show that $\Psi|_A$ is pure. Let $I$ be a face of $\Psi|_A$ and let $B$ be a basis that contains $F$. By QE applied to $B$ and $B_0$ it is possible to replace all vertices in $B\backslash A$ with elements of $B_0$, since $\min E-A > \max B_0$. The resulting basis is contained in $A$ and it still contains $F$. It follows that $\Psi|_A$ is pure. 

For [iv.] notice that if $B_1$ and $B_2$ are bases of $\Psi\slash v$ then $B_1\cup\{v\}$ and $B_2\cup\{v\}$ are bases of $\Psi$ and the QE axiom applies. Notice that in this case $v$ is irrelevant since it belongs to both bases, thus the QE axiom holds in $\Psi\slash v$. 
\end{proof}

The QE axiom is pretty rich from the perspective of combinatorial topology in particular, by using shellability. While Theorem~\ref{thm:PURE} shows that $\Psi$ is shellable, the shelling orders provided by the vertex decomposition are obtained recursively, and consequently, the restriction sets are difficult to study. As the generating function of the restriction is equal to the $h$-polynomial of $\Psi$, it is desirable to have shelling with restriction sets that are easier to compute directly. Complexes in QE guarantees that this holds.

\begin{thm}\label{thm:lexShell} Let $\Psi = (E, \Delta)$ be a complex in QE. The lexicographic order on $\B$ is a shelling order. Under this shelling order $\RR(B) = IP(B)$.
\end{thm}
\begin{proof} Let $B_1 <_{lex} B_2$ be bases of $\Psi$. The goal is to find a basis $B_3 <_{lex} B_2$ such that $B_1\cap B_3 \subseteq B_2\cap B_3 = B_2\backslash \{e\}$. The proof is by induction on $r-|B_1\cap B_2|$. The base case is trivial. There are two cases to consider: 
\begin{compactenum}
\item[Case 1.] The element $b:=\max B_1\triangle B_2$ is an element of $B_2$. By QE there is $b'\in B_1 \backslash B_2$ such that $B_3 = (B_2\backslash\{b\})\cup\{b'\}$ is a basis. That choice of $B_3$ works directly. 
\item[Case 2.] The element $b:=\max B_1 \triangle B_2$ is an element of $B_1$. By the QE axiom there is $b' \in B_2 \backslash B_1$ such that $B_1' = (B_1\backslash\{b\})\cup\{b'\}$ is a basis. Then $B_1\cap B_2 \subset B_1'\cap B_2$ and by inductive hypothesis, there exists $B_3$ such that $B_1'\cap B_2 \subseteq B_2\cap B_3 = B_2 \backslash\{e\}$, and so such $B_3$ does the job. 
\end{compactenum}

The second statement follows directly from the definition of the restriction set exactly as in \cite{MR1165544}.
\end{proof}

\begin{rem} Theorem~\ref{thm:lexShell} shows that the class QE is contained in LEX and is much smaller in general. Graphs are rank-two simplicial complexes and shellable just means that they are connected. Given any connected graph $\Delta$ with with vertex set $[n]$ (viewed as a rank-two simplicial complex) such that the graph $\Delta|_{[r]}$ is connected for every $1\le r \le n$, we get that the complex $([n], \Delta)$ is an element of LEX. However a substantial proportion of such graphs in not in QE. Therefore, QE is significantly smaller than LEX even if we just compare rank-two complexes.\end{rem}

The following corollary is an immediate consequence of Theorem~\ref{thm:lexShell}. It generalizes the result of Bj\"orner \cite{MR1165544} and, as we will see in Section~\ref{sct:Stanley}, it becomes more powerful when studied in this context. 

\begin{cor}\label{cor:QEh-poly} If $\Psi = (E, \Delta)$ is in LEX (or in QE), then \begin{equation} h(\Psi, x) = \sum_{B\in \B} x^{|IP(B)|} \end{equation}
\end{cor}

There is another remarkable way to obtain complexes in QE from older ones which is less conventional in combinatorial topology. It implies that the partial steps of the construction of a complex in QE using the shelling of \ref{thm:lexShell} are complexes in QE. 

\begin{thm}\label{thm:QEGale} Let $\Psi = (E, \Delta)$ be a complex in QE and let $\J$ be an order ideal of $\Gale(\Psi)$. Then $\Psi[\J]$ is in QE. Furthermore, for every basis $B$ of $\Psi[\J]$, the equality $IP(B, \Psi[\J]) = IP(B, \Psi)$ holds. 
\end{thm}
\begin{proof} The first part of the theorem follows directly from the fact that the lexicographic order is a linear extension of $\Gale(\Psi)$ and the fact that an exchange produced by the QE axiom yields a new basis that is smaller in the lexicographic order than the original one. 

The equality of passive sets even holds for general ordered complexes since the bases that witness external activity of $B$ are all smaller than $B$ in $\Gale(\Psi)$ (they differ from $B$ by one element). 
\end{proof}

A useful property of complexes in QE concerns the strucure of $\Gale(\Psi)$. 

\begin{lemma}\label{lemma:minGale} Let $\Psi = (E, \Delta)$ be a complex that satisfies QE. Then $B_0$ is the unique minimal basis of $\Gale(\Psi)$. In particular, QE is a sublcass of GALE. 
\end{lemma}

\begin{proof} It is straightforward that $B_0$ is minimal, since the lexicographic order is a linear extension of $\Gale(\Psi)$. On the other hand, if $B$ is any other basis, apply the QE axiom with $B_0$ to get a basis $B'$. The exchange takes an element from $B$ and replaces it with a smaller vertex, thus $B'<_{\Gale} B$. 
\end{proof}

Another useful tool is the internal poset of Las Vergnas. We will see that $\Int(\Psi)$ is coarser than $\Gale(\Psi)$ which will be handy when we study Stanley's conjecture. The following are some structural results.

\begin{lemma}\label{lemma:Grading} Let $\Psi$ be a complex in QE and let $B$ be a basis. There exists an element $b\in IP(B)$ and  a basis $B'$ such that $IP(B') = IP(B)\backslash \{b\}$. In particular, $\Int(\Psi)$ is graded and the degree of a basis is the cardinality of its internally passive subset.  
\end{lemma} 
\begin{proof}
Since the lexicographic order is a shelling and since the restriction sets that shellings are the the internally passive sets, the first part of the result follows from \cite[Lemma 7.2.6]{MR1165544}. The fact that the cardinality of the passive set provides a grading of $\Int(\Psi)$ is straightforward from the previous part. 
\end{proof}

Finally, we verify that there is the following relationship between the Gale and the Int posets of complexes in QE. 

\begin{thm}\label{thm:intVs.Gale} Let $\Psi$ be a complex in QE and let $B$ and $B'$ be bases such that $IP(B) \subseteq B'$. Then $B<_{\Gale}B'$ and hence $\Gale(\Psi)$ is a poset extension of $\Int(\Psi)$. 
\end{thm}
\begin{proof} Fix $B$ and proceed by induction on $|B'\backslash B|$. The base case is $B= B'$, in which case there is nothing to show. Assume that the property holds for all bases $B''$ with $IP(B) \subseteq B''$ and $|B''\backslash B| < |B' \backslash B|$. Next apply the QE axiom with $B$ and $B'$. Let $u = \max{B\triangle B'}$. Notice that $u\notin IP(B) \subseteq B\cap B'$. Hence if $u\in B$, then by the QE axiom there would be $b'\in B'\backslash B$ such that $B'' =(B\backslash\{u\}) \cup \{b'\}$ is a basis, but in this case $B'' <_{lex} B$ and $IP(B) \subseteq B''$ which is impossible. It follows that $u\in B'$ which in turns implies the existence of an element $b \in B$ such that $B'':=(B'\backslash \{u\}) \cup \{b\}$ is a basis. Then $B'' <_\Gale B'$, $IP(B) \subseteq B''$ and $|B''\backslash B| < |B'\backslash B|$. By the inductive hypothesis $B<_\Gale B'' <_\Gale B'$ as desired. 

If $B$ and $B'$ are bases with $B$ smaller than $B'$ in $\Int(\Psi)$, then $IP(B) \subseteq IP(B') \subseteq B'$, thus $B<_{Gale}B'$, which shows that $\Gale(\Psi)$ is a poset extension of $\Int(\Psi)$. 
\end{proof}

\subsection{The Circuit axiom}\label{ssct:circ}
The theory of circuits in matroid theory is what allows for meaningful external activity theories to play a prominent role in the understanding of objects such as the broken circuit complex and Orlik-Solomon algebras. It is a dual notion to that of internal activity in matroid theory. Unfortunately, this is not the case in the quasi-matroidal setting and a careful separate study is required when it comes to duality. The first step toward understanding QC is, again, constructing new complexes from old ones. 
\begin{thm}  
Let $\Psi= (E, \Delta)$ and $\Psi'=(E', \Delta')$ be ordered complexes in QC. 
\begin{compactenum}
\item[i.] If $s(E, E')$ is a shuffle of $E$ and $E'$, then the join $\Psi*_{s(E, E')} \Psi'$ is in QC. 
\item[ii.] If $U\subset E$ is any subset, then $\Psi|_U$ is in QC.  
\item[iii.] If $v$ is a vertex of $\Psi$, then contraction $\Psi \slash v$ is in QC.
\end{compactenum} 
In particular, QC is a quasi-matroidal class. 
\end{thm}

\begin{proof}
Part [i.] follows easily from the fact that circuits in the join are either circuits in $\Psi$ or in $\Psi'$, thus if they intersect, they come from the same complex. 

Part [ii.] is an easy consequence of the fact that the circuits of $\Psi|_U$ are the circuits of $\Psi$ that are contained in $U$. 

Part [iii.] follows since the circuits of $\Psi\slash v$ are in natural bijection with circuits of $\Psi$ that contain $v$. 
\end{proof}

\begin{thm}\label{thm:QCFundCirc} Let $\Psi = (E, \Delta)$ be a complex in QC. If $B$ is a basis and $e\in E\backslash B$ is externally active, then there is a unique circuit contained in $B\cup \{e\}$. 
\end{thm}
\begin{proof} 
Assume there are two such circuits $C_1$, $C_2$. Since $e$ is externally active we may assume that $e= \min C_1$. Notice that $e$ is also in $C_2$, because the circuit cannot be a subset of $B$. Then $c<\max C_1\cap C_2$ and the QC axiom implies that there is a circuit $C_3\subseteq (C_1\cup C_2)\backslash \{e\} \subseteq B$. 
\end{proof}

Furthermore, the independent sets of $\nbc(\Psi)$ also have a simple description in terms of broken circuits. 

\begin{lemma} Let $\Psi = (E , \Delta)$ be a pure complex in QC. An independent set $I$ of $\Psi$ is independent in $\nbc(\Psi)$ if and only it contains no broken circuit. 
\end{lemma}

\begin{proof} Since the bases of $\nbc(\Psi)$ contain no broken circuit, any independent set in $\nbc(\Psi)$ does not contain a broken circuit either. Thus it suffices to show that if $I\in \Delta$ contains no broken circuit, then $I$ is independent in $\nbc(\Psi)$. Let $B$ be the smallest lexicographic basis of $\Psi$ that contains $I$ and assume that $B$ contains a broken circuit $\tilde C$ where $C = \tilde C \cup \{c\}$ is the circuit that was broken. There is an element $c' \in \tilde C \backslash I$. By Theorem~\ref{thm:QCFundCirc}, $C$ is the unique circuit contained in $B\cup\{c\}$. Hence $B\backslash\{c'\}\cup\{c\}$ is a basis that is smaller in the lexicographic order and contains $I$, leading to a contradiction. 
\end{proof}

This endows QC with a theory of fundamental circuits analogous to that of matroids. That is, if $B$ is a basis and $e\in EA(B)$ then the fundamental circuit $Ci(B,e)$ is the unique circuit contained in $B\cup \{e\}$.  Fundamental circuits play a key role in studying various aspects of the Tutte polynomial, which will be studied in a quasi-matroidal setting in Section 6.

\section{The First Basis Property}
The goal of this section is to introduce an important property that holds for matroids and shifted complexes and imposes additional structure on Gale truncations. It is a local condition for ordered complexes. Recall that for a vertex $v$ of an ordered complex, $B_{v,0}$ denotes the smallest lexicographic basis of $\Psi\slash v$. 

\begin{defn}\label{defn:facetProp} Let $\Psi = (E, \Delta)$ be an ordered complex and let $B_0$ be the first lexicographic basis. We say that $\Psi$ satisfies the First Basis Property (FBP) if either:
\begin{compactitem} 
\item[i.] the rank of $\Psi$ is $1$, or
\item[ii.] $\Psi$ has exactly one basis, or 
\item[iii.] for every vertex $v$ of $\Psi$ that is not in $B_0$, the contraction $\Psi\slash v$ satisfies FBP and $B_{v,0}\subset B_0$. 
\end{compactitem}
\end{defn}

Notice that shifted complexes satisfy the FBP. We now show that matroids satisfy the FBP. For this, we first recall the following result:
\begin{thm} \emph{\cite[Corollary 2.3]{Dall}}\label{thm:Dall} Let $\Psi = (E, \Delta)$ be an ordered matroid. Then $IA(B) \subseteq B_0$ for every basis $B$. \end{thm}

This allows us to show the first main result of the section. 

\begin{thm}\label{thm:FBPMatroids} Ordered matroids satisfy the FBP.
\end{thm}
\begin{proof} The proof is a direct induction on the rank of the matroid. The base case is trivial. Let $\Psi$ be a rank-$r$ ordered matroid and let $v$ be a vertex of $\Psi$ not in $B_0$. Then $\Psi\slash v$ is a matroid of a smaller rank, thus it satisfies the FBP by induction. Now let $B$ be the smallest lexicographic facet that contains $v$. Then $IP(B) \subseteq \{v\}$ and since $IP(B)$ is the restriction set of $B$ with respect to the lexicographic shelling order, it follows that $|IP(B)| >0$, since $B \not= B_0$. Thus $IP(B) = \{v\}$ and $B\backslash\{v\} = IA(B) \subseteq B_0$ by Theorem~\ref{thm:Dall}.
\end{proof}

As in the previous section the next step is to study how to construct new complexes from old. The proof is straightforward and is ommitted. 

\begin{thm}\label{thm:FBPconst} Let $\Psi = (E, \Delta)$ and $\Psi' = (E', \Delta')$ be complexes that satisfy the FBP. 
\begin{compactenum}
\item[i.] If $A\subseteq E$ contains $B_0$ and  $\Psi|_A$ is pure, then the restriction $\Psi|_A$ satisfies the FBP.
\item[ii.] If $I$ is an independent set of $\Psi$, then the contraction $\Psi\slash I$ satisfies the FBP. 
\item[iii.] If $s(E, E')$ is a shuffle of $E$ and $E'$, then the join $\Psi*_{s(E,E')}\Psi'$ satisfies the FBP.
\item[iv.] If $\cal J$ is an order ideal of $\Gale(\Psi)$, then the Gale truncation $\Psi[\cal J]$ satisfies the FBP.
\end{compactenum}
\end{thm}
\begin{rem} Notice that Theorems~\ref{thm:FBPMatroids} and \ref{thm:FBPconst} imply that if $\A$ is a quasi-matroidal class, then the class $\A\cap$ FBP of complexes in $\A$ that satisfy FBP is also a quasi-matroidal class. 
\end{rem}
We will study the classes QE$\cap$FBP, QI$\cap$FBP and their intersection. 

\begin{lemma}\label{lemma:contraction} Assume that $\Psi= (E, \Delta)$ is a complex that satisfies the FBP and let $I$ be an independent set such that $B_0\cap I = \emptyset$. Then $B_{I,0}\subseteq B_0$. 
\end{lemma}
\begin{proof} If $v\in I$ then  $\Psi\slash I = (\Psi \slash v) \slash (I\slash\{v\})$. Then by induction and the FBP condition $B_{I,0} \subseteq B_{v,0} \subseteq B_0$. 
\end{proof}

\begin{thm}\label{thm:QE+FPP} Let $\Psi = (E, \Delta)$ be a complex  QE$\cap$FBP, and let $B$ be a basis of $\Psi$. Then $B\backslash B_0\subseteq IP(B)$. 
\end{thm}
\begin{proof} The proof goes by induction on the rank of $\Psi$. Let $v$ be the maximal vertex in $I:= B\backslash B_0$. By QE applied with $B_0$ we obtain that $v\in IP(B)$. From $B_{v,0}\subset B_0$ it follows that $I-{v}\subseteq (B\backslash\{v\}) \backslash B_{v,0}$. Thus for $v'\in I\backslash \{v\}$ and by induction $v' \in IP(B\backslash v, \Psi\backslash v)$. That is, there is $u< v'$ such that $B\backslash\{v,v'\} \cup \{u\}$ is a basis of $\Psi \backslash v$. Therefore $B\backslash\{v'\}\cup \{u\}$ is a facet of $\Psi$, and so $v' \in IP(B)$.
\end{proof}

Finally, the following result will provide flexibility in doing inductive arguments on the number of facets of complexes and in exploiting the whole power of the shelling orders. 

\begin{thm}\label{thm:QE+QI+FBP} Let $\Psi = (E, \Delta)$ be a complex in QI$\cap$FBP and let $\cal J$ be an order ideal of $\Gale(\Psi)$. Then $\Psi[\cal J]$ satisfies QI and FBP. 
\end{thm}

\begin{proof} Since we saw that the FBP is preserved by Gale truncations, it suffices to show that $\Psi[\cal J]$ is in QI. Let $I_1$ and $I_2$ be independent sets in $\Psi[\cal J]$ with $|I_1|>|I_2|$ and $I\subseteq I_1\cap I_2$ such that $I_1\backslash I \subseteq B_{I, 0}$. Notice that $B_{I_2, 0}$ is a basis of $\Psi[\cal J]$, because $I_2$ is independent in $\Psi[\cal J]$. By FPP $B_{I_2,0} \subseteq B_{I,0}$. It follows from the cardinality condition and the pigeon hole principle there is some $ e\in (I_1\backslash I_2) \cap B_{I_2,0}$. 
\end{proof}

\begin{rem} Notice that QI is, in general, not closed under Gale truncations. For example, the complex $\Psi_2$ in the proof of Theorem~\ref{thm:ex} fails to satisfy the FBP and it is easily seen that removing the top Gale facet makes the QI axiom fail. 
\end{rem}
\section{Tutte polynomials and nbc complexes}\label{sec:TuttePolynomialNBCComplex}\label{sct:Tutte} In this section we develop a theory of Tutte polynomials (that is, a universal deletion-contraction invariant)  for complexes in QE$\cap$QC. The aim is to get a theory of Tutte-Grothendieck invariants similar to that in matroid theory. 

Let $\cal S := $ QE$\cap$QC and let $R$ be a ring. An invariant $f$ is a map that associates to every complex $\Psi$ of $\cal S$ an element $f(\Psi) \in R$ in such a way that if $\Psi \cong \Psi'$, then $f(\Psi) = f(\Psi')$. A \emph{Tutte-Grothendieck invariant}, \emph{ TG-invariant} for short, is an invariant that satisfies the following recurrence:
\begin{equation}\label{eqn:TG} f(\Psi) = \begin{cases} f(\Psi|_{\{e\}})f(\Psi|_{E\backslash\{e\}}) \, \text{ if } e \text{ is a loop or a coloop,} \\ f(\Psi\slash \{e\}) + f(\Psi\backslash \{e\}) \, \text{    if  } e \text{ is the largest non-coloop vertex of $\Psi$.}\\ \end{cases} \end{equation}

Just as in matroid theory are many natural TG invariants for the class $\cal S$ and there is a TG invariant that rules them all. 


\begin{defn} Let $\Psi$ be a complex in $\cal S$. The Tutte Polynomial of $\Psi$ is defined to be
\begin{equation}T(\Psi, x, y ) := \sum_{B\in \B} x^{|IA(B)|} y^{|EA(B)|}. \end{equation}
\end{defn}

\begin{thm} The Tutte Polynomial is a TG invariant that is universal in the following sense: if $f$ is a TG-invariant on the class $\cal S$, then for every $\Psi$ in $\cal S$ we have that  \begin{equation} f(\Psi) = T(\Psi, f(\Psi_{\emph{\text{coloop}}}),f( \Psi_{\emph{\text{loop}}})).\end{equation}
\end{thm}
\begin{proof} Notice that $T(\Psi_{\text{coloop}}, x, y) = x$ and $T(\Psi_{\text{loop}}, x, y ) = y$. Every loop is externally active and every coloop is internally active. Thus the recurrence works for loops and coloops. If $v$ is the largest non-coloop, then the facets are divided into two types. 
\begin{compactenum}
\item[a.] $v\notin B$, that is, $B$ is a basis of $\Psi \backslash v$. Then $v$ is an internally passive element of $B$, and so $IA(B, \Psi) = IA(B, \Psi\backslash v)$. Circuits of $ \Psi\backslash v$ are circuits of $\Psi$ that do not contain $v$, thus $EA(B, \Psi) = EA(B, \Psi\backslash v)$. 
\item[b.] $v \in B$, that is, $B\slash \{v\}$ is a basis of $\Psi\slash v$. Then $v$ is internally passive in $B$ by QE with $B_0$ and it is straightforward that $IA(B, \Psi) = IA(B\slash \{v\}, \Psi \slash v)$. If $e\in E\backslash (B\backslash \{v\})$, then $v$ is not the smallest element of any circuit in $B\cup\{e\}$: $e$ is in any such circuit and if it is not a coloop then $e<v$. Hence $EA(B, \Psi) = EA(B\slash\{v\}, \Psi\slash v)$.
\end{compactenum}

The result follows directly. The fact that evaluations of the Tutte polynomial gives rise to all TG invariants is a straightforward inductive argument. 
\end{proof}

Next we provide a wealth of interesting invariants that satisfy such a recursion. The most prominent one comes from the theory of nbc complexes. They turn out to behave pretty similar to matroids. Recall from Corollary~\ref{cor:QEh-poly} that the $h$-polynomial is given by $h(\Psi, x) = \sum_{i=0}^r h_ix^i=\sum_{B \in \B} x^{|IP(B)|}$ for any rank-$r$ complex $\Psi$ in QE. If it also satisfies QC it is possible to compare the $h$-polynomial with the evaluation $T(\Psi, x, 1)$. 

\begin{lemma} The $h$-vector is a TG invariant for the class $\cal S$ the following sense: if $\Psi$ is a complex in $\cal S$ then
\[h(\Psi, x) = x^rT(\Psi, x^{-1}, 1)\]
\end{lemma}
\begin{proof} Use the equations $|IP(B)| + |IA(B)| = r$ and $T(\Psi, x, 1) = \sum_{B\in \B} x^{|IA(B)|}$. 
\end{proof}

This lemma implies that standard objects such as the (reverse) $f$-polynomial, the number of faces and the number of bases are evaluations of the Tutte polynomial. They are, however, not very surprising. A much stronger and interesting evaluation that has a rich combinatorial interpretation comes from the theory of nbc complexes. The characteristic polynomial of $\Psi$ can be defined as the standard evaluation of the Tutte polynomial and one might hope that there is a natural poset that replaces the lattice of flats to get the M\"obius function interpretation of the characteristic polynomial. This question is particularly interesting in the class of shifted complex and it is the source of inspiration for the definition of another quasi-matroidal class.  

\begin{thm} Let $\Psi = (E, \Delta)$ be a complex in $\cal S$. Then the lexicographic order of the bases of $\nbc(\Psi)$ is a shelling order of $\nbc(\Psi)$. 
\end{thm}
\begin{proof} The first paragraph of Bj\"orner's argument in \cite[Lemma 7.3.2]{MR1165544} is replaced by Theorem~\ref{thm:lexShell}. The second paragraph applies verbatim to get the result. 
\end{proof}

Now we show that the polynomial $f(\Psi, x) = \sum_{B \in \B(\nbc(\Psi))} x^{|IA(B, \nbc(\Psi))|}$ is a TG invariant. It is easy check that $f(\Psi_{\text{coloop}}, x) = x$ and $f(\Psi_{\text{loop}}, x)=0$. This implies that  $f(\Psi, x) = T(\Psi, x, 0)$ just as in the classical case. 

\begin{thm}\label{thm:nbcTutte} The polynomial invariant $f$ is a TG invariant. Consequently, if $\Psi$ is a complex in $\cal S$ and $B$ is an nbc basis of $\Psi$, then $IA(B, \Psi) = IA(B, \nbc(\Psi))$. 
\end{thm}
\begin{proof}
If there is a loop $v$, then $\nbc(\Psi)$ is the void complex since $\emptyset$ is a broken circuit; since $f(\Psi_{\text{loop}}) = 0$, the recurrence holds directly. If $v$ is a coloop, then it is externally active for every basis, thus $IA(B, \nbc(\Psi)) = \{v\}\cup IA(B\backslash \{v\},\nbc(\Psi)|_{E-v})$. Assume then that $\Psi$ has no loops (otherwise $\nbc(\Psi)$ is trivial).

Let $v$ be the largest non-coloop vertex of $\Psi$ and let $B$ be a basis of $\nbc(\Psi)$. There are two cases: 
\begin{compactenum}
\item[Case 1.] $v\not\in B$. Then $u \in IA(B, \nbc(\Psi))$ if and only if $B$ is the smallest lexicographic basis that contains $B\backslash\{u\}$. This, however, is equivalent to a statement in $\Psi \backslash v$ since adding $v$ to $B\backslash\{u\}$ yields a basis bigger than $B$ in lexicographic order. Thus $IA(B, \nbc(\Psi)) = IA(B, \nbc(\Psi\backslash v))$. 

\item[Case 2.] $v\in B$. First we claim that $v \in IP(B, \nbc(\Psi))$. Note that $v\in IP(B, \Psi)$, because it is the largest vertex and we can use the quasi-exchange axiom with $B_0$ and $B$. Let $B'$ be the smallest lexicographic basis of $\Psi$ that contains $B\backslash v$. We claim that $B'$ is a basis of $\nbc(\Psi)$. Note that $v\notin B'$ since applying QE with $B$ and $B_0$ allows us to remove $v$ to get a smaller basis. Let $b\in B'\backslash B$ and assume, for the sake of contradiction, that there is a broken circuit $\gamma = C-c \subseteq B'$. Since $B$ is an nbc basis, $\gamma$ is not contained in $B$, thus $b\in C-c$ and $c<b$. Notice that $c$ is externally active and hence $C$ is the unique circuit in $B'\cup C$ by Theorem~\ref{thm:QCFundCirc}. Therefore $(B'\backslash \{b\})\cup \{c\}$ is a basis that contains $B\backslash \{v\}$ and is smaller in lexicographic order than $B'$, which contradicts the choice of $B'$.

Next we notice that $B\backslash \{v\}$ is a basis of $\nbc(\Psi\slash v)$. To prove this assume that $u\in EA(B\backslash \{v\}, \Psi\slash v)$ and let $C$ be the unique circuit of $\Psi \slash v$ contained in $(B\backslash \{v\})\cup\{u\}$. Then one out of $C$ and $C\cup\{v\}$ is a circuit in $\Psi$. Note that $u$ is a vertex of $\Psi$ and since it belongs to a circuit, it is not a coloop. It follows that $u$ is the smallest element in $C\cup\{v\}$ and hence in the corresponding circuit $\hat C$ that contains it. This circuit $\hat C$ is contained $B\cup\{u\}$, and so $u\in EA(B, \Psi)$. This is a contradiction, since $B$ is an nbc basis in $\Psi$.

Furthermore, notice that if $B'$ is a basis of $\nbc(\Psi\slash \{v\})$, then $B'\cup \{v\}$ is a basis of $\Psi$. If there is some element $u$ in $EA(B'\cup\{v\}, \Psi)$, then there is a circuit $C\subseteq B'\cup \{u, v\}$ for which $u$ is the minimal element and $C\cap B'$ is a circuit in $\Psi\slash v$ for which $u$ is the minimal element, contradicting the fact that $B'$ is an nbc basis of $\Psi\slash v$. 

The previous two paragraphs imply that $\nbc(\Psi \slash v ) = \nbc(\Psi) \slash v$. Now notice that  $u \in IA(B, \nbc(\Psi))$ if and only if $B$ is the smallest lexicographic basis of $\nbc(\Psi)$ containing $B\backslash\{u\}$, and since $v\in B\backslash\{u\}$ this is equivalent to saying that $B\backslash\{v\}$ is the smallest lexicographic basis of $\nbc(\Psi)\slash v = \nbc(\Psi\slash v)$ that contains $B\backslash\{u,v\}$. This, in turn, equivelent to saying that  $u$ is an element of $IA(B\backslash \{v\}, \nbc(\Psi\slash v))$.
\end{compactenum}

Computing the activities polynomial yields: 
\begin{eqnarray*} f(\Psi, x)  &= & \sum_{B\in \B(\nbc(\Psi))} x^{|IA(B,\nbc(\Psi))|} \\ 
&=& \sum_{v\in B\in \B(\nbc(\Psi))} x^{|IA(B,\nbc(\Psi))|} + \sum_{v\notin B \in \B(\nbc(\Psi))} x^{|IA(B,\nbc (\Psi))|} \\ 
&=& \sum_{B'\in \B(\nbc(\Psi\slash v))} x^{|IA(B',\nbc(\Psi\slash v))|} + \sum_{ B''\in \B(\nbc(\Psi\backslash v))} x^{|IA(B'', \nbc(\Psi\backslash v))|} \\ 
&=& f(\Psi\slash v, x) + f(\Psi\backslash v, x).
\end{eqnarray*}

It follows that $f(\Psi, x) = T(\Psi, x, 0)$. Notice that $T(\Psi, x, 0 ) = \sum_{B\in \B, EA(B) = \emptyset} x^{|IA(B, \Psi)|}$. Also, it is straightforward that $IA(B, \Psi) \subseteq IA(B, \nbc(\Psi))$ for every basis of $\nbc(B)$ and hence, by the polynomial equality, $IA(B, \Psi) = IA(B, \nbc(\Psi))$
\end{proof}

In standard terms the previous theorem can be rewritten as follows: 

\begin{cor}\label{cor:nbc-hvect}
If $\Psi$ is a rank-$r$ complex in QE$\cap$QC, then \begin{equation}\label{eqn:nbc-hvect}h(\nbc(\Psi),x) = x^rT(\Psi,x^{-1}, 0).\end{equation}
\end{cor}
\section{A refinement of Stanley's conjecture}\label{sct:Stanley}
The purpose of this section is to propose an extension of Stanley's $h$-vector conjecture to the setting complexes in QI$\cap$QE$\cap$FBP. The lexicographic order of the bases of an ordered matroid has a family of well understood restriction sets whose size generating function gives the $h$-vector of the independence complex. Techniques from combinatorial topology suggest that a plausible approach to understanding the properties of such structures is by recursive construction adding one basis at a time. However, removing bases from the matroid, even in the correct order suggested by the shelling, produces complexes that are not matroidal. One of the advantages of working in the more general setting is that this problem disappears. Removing facets consistently with the shelling order preserves the validity of the axioms and gives tools to prove theorems by induction on the number of bases.



For the remainder of this section, unless stated otherwise, $\Psi = (E, \Delta)$ denotes a rank-$r$ complex in QE$\cap$QI$\cap$FBP. The goal is to mimic conjecture  3.10 from \cite{MR3339029} and build a finer conjecture in this larger class. In order to do that it is necessary to introduce some extra notation. Fix a subset $A$ of $E$ and consider the splitting of an independent $I$ of $\Delta$ in the two sets $I\cap A$ and $I\backslash A$. This induces a partition of the independent sets according to their part that is not in $A$. For any $F\in \Delta|_{E-A}$, let $\Psi_{A,I}:= (\Psi\slash I)|_A = (A, \Delta_{A,I})$. The set of independents of $\Delta$ is the (disjoint) union of the independents of the $\Delta_{A,I}$ sets and writing this fact in terms of $h$-polynomials we derive the following lemma.

\begin{lemma}Let $\Psi = (E,\Delta)$ be an arbitrary pure ordered complex. For an arbitrary $A\subset E$, we have that \begin{equation}h(\Psi, x) = \sum_{I\in \Delta|_{E\backslash A}} (1-x)^{(d-|I|) - \rk\,  \Psi_{A,I}}h(\Psi_{A,I}, x).\end{equation}
In particular, if $\rk\, \Psi_{A,I} = d-|I|$ for every $I\in \Delta|_{E\backslash A}$, then \begin{equation}\label{decomposition}h(\Psi, x) = \sum_{I\in \Delta|_{E\backslash A}}x^{|I|} h(\Delta_{A,I},x).\end{equation}
\end{lemma}
\begin{proof}
The following identity for the $f$-polynomial follows from the splitting of the faces: 
\begin{equation}f(\Psi, x) = \sum_{I\in \Delta|_{E\backslash A}} x^{|I|}f(\Psi_{A,I}, x). \end{equation}
Transforming this into an $h$-vector equality yields the result. 
\end{proof}
If the complex $\Psi$ is either in QE or QI, then the refined decomposition of equation (\ref{decomposition}) holds in many cases: 
\begin{cor}\label{coro:h-vector} The refined decomposition of the $h$-vector holds for $\Psi$ and $A$ in the following cases: 
\begin{compactenum}
\item[i.] $\Psi$ is in QE and $A$ contains all coloops and all elements smaller than or equal to the smallest non-coloop of $B_0$. 
\item[ii.] $\Psi$ is in QI and $A$ contains $B_0$. 
\end{compactenum} 
\end{cor}

\begin{proof}
For each case it suffices to show that, for any independent set $I$ disjoint from $A$, there is a basis $B$ containing $I$ such that $B\backslash I \subseteq A$: it implies that the restricted contractions have the correct rank.  \begin{compactenum}
\item[i.] If $B'$ is a basis that contains $I$, then every element of $B$ that is not contained in $A$ or $I$ allows to apply QE with $B'$ and $B_0$ to obtain $B$. 
\item[ii.] Apply QI with $B_0$ and $I$ to obtain $B$. 
\end{compactenum}
\end{proof}


A subset $A$ that satisfying the conditions of Corollary~\ref{coro:h-vector} is called \emph{admissible}. The following theorem gives a combinatorial interpretation of Corollary \ref{coro:h-vector}.
\begin{thm} Assume that $\Psi = (E, \Delta)$ is in QE  and $A$ is an admissible subset of $E$. 
If $B$ is a basis with $B\backslash B_0 = I$, then  
\begin{equation}\label{passive} IP(B) = I\cup IP(B\backslash I,\Psi_{A,I}). \end{equation} Finally, if $\Psi$ is also in QI$\cap$FBP, and if $A$ contains $B_0$, then equation (\ref{passive}) holds as well. 
\end{thm}
\begin{proof} Notice first that in both cases, $I\subseteq IP(B)$: If $\Psi$ satisfies only QE then every element of $I$ can be exchanged using QE with $B$ and $B_0$. If $\Psi$ also satisfies QI and FBP this is simply Theorem~\ref{thm:QE+FPP}. 

Thus in both cases it follows that $IA(B, \Psi) = IA(B,\Psi)\backslash A \subseteq IA(B\backslash I, \Psi_{A,I})$: if $b\in IA(B, \Psi)$ then $B$ is the smallest lexicographic basis of $\Psi$ that contains $B\backslash\{b\}$, and so $B\backslash I$ is the smallest lexicographic basis of $\Psi\slash I$ that contains $\B\backslash(I\cup\{b\})$. 

Passing to the passive sets yields that $F\cup IP(B\backslash I, \Psi\slash I) \subseteq IP(B, \Psi)$. The equality follows by the double computation of the $h$-polynomial of $\Psi$ using Corollaries~\ref{cor:QEh-poly} and \ref{coro:h-vector}, i.e., \begin{equation}\label{eqn:restrictionSets} h(\Delta, x) = \sum_{B\in \B} x^{|IP(B)|} =  \sum_{I\in \Delta|_{E\backslash A}}x^{|I|} h(\Delta_{A,I},x). \end{equation}
Both sums can be regarded as a sum over the bases and the inclusion of passive sets above forces all the equalities. 
 \end{proof}

Similar results were used in \cite{MR3339029} to provide a combinatorial conjecture that implies Stanley's conjecture for matroids. The next goal is to extend the combinatorial conjecture to the class of complexes satisfying QE, QI and FBP. It becomes feasible to do induction on the number of bases by taking order ideals of the Gale posets. Such order ideals preserve QE, QI and FBP, but general order ideals of matroids are not matroids. We expect that this technique will have many applications and hope it could eventually lead to a full resolution of Stanley's conjecture.  

The main idea for the stronger conjecture is to use the $h$-vector decomposition and the relations between $\Int(\Psi)$ and $\Gale(\Psi)$. Recall, from Lemma~\ref{lemma:Grading}, that $\Int(\Psi)$ is graded by the size of the internally passive sets, hence its rank generating function coincides with the $h$-polynomial of $\Psi$. One way to attack Stanley's conjecture is to construct one monomial of degree $|IP(B)|$ for every basis of an ordered matroid to obtain a suitable multicomplex from a combinatorial rule. A candidate for the poset of such a multicomplex under divisibility could be $\Int(\Psi)$. This, however, does not work in general: $\Int(\Psi)$ fails to have enough relations to be the face poset of a multicomplex. However, one might hope that it is possible to add some extra relations to $\Int(\Psi)$ to obtain the face poset of a multicomplex with the same rank generating function. This is the point where the inductive step should come: by Theorem~\ref{thm:intVs.Gale} the Int poset of any Gale truncation of $\Psi$ is an order ideal of $\Int(\Psi)$. 

The following conjecture is a strengthening of Stanley's conjecture about $h$-vectors of matroids. It aims to construct the multicomplex using the very rich combinatorial structure of complexes satisfying QE, QI and FBP. 
\begin{conj}\label{conj:main} There exists a map $\cal F$ from the class of ordered simplicial complexes in QE$\cap$QI$\cap$FBP to the class of multicomplexes, such that for each $\Psi$ the multicomplex ${\cal F}(\Psi)$ satisfies the following: 
\begin{compactenum}
\item[i.] The set of variables appearing in ${\cal F}(\Psi)$ is $\{x_i : i \text{ is a vertex of }\Delta \text{ not contained in } B_0\}$. 
\item[ii.] The monomials of ${\cal F}(\Psi)$ are in bijection with bases of $\Psi$, in such a way that: 
\begin{compactenum}
\item the monomial associated to the basis $B$ under this bijection is denoted by $m_B$,
\item the degree of $m_B$ is equal to $|IP(B)|$,
\item the support of $m_B$, i.e., the set $\{e\in E \, :  \, x_e|m_B\}$, is equal to $B\backslash B_0$.
\end{compactenum}
\item[iii.] The poset ${\cal M}(\Psi) = (\B, <_m)$,  where we write $B<_m B'$ if and only if $m_B|m_{B'}$, is an extension of $\Int(\Psi)$ and is extended by $\Gale(\Psi)$. 
\item[iv.] If $A\subseteq E$ contains $B_0$, then ${\cal F}(\Psi|_A) \subseteq {\cal F}(\Psi)$. 

\end{compactenum}
\end{conj}

Roughly speaking, the conjecture predicts that there is a natural way to extend the poset $\Int(\Psi)$ by adding some relations of $\Gale(\Psi)$. We now show that this conjecture in fact implies Stanley's conjecture. 

\begin{thm}\label{thm:ImpStan} Conjecture~\ref{conj:main} implies Stanley's conjecture.
\end{thm}

\begin{proof} Assume conjecture~\ref{conj:main} holds and let $\Psi$ be an ordered matroid with multicomplex ${\cal F}(\Psi)$. We have to show that ${\cal F}(\Psi)$ is pure. ${\cal M}(\Psi)$ is an extension of $\Int(\Psi)$ that preserves the grading and all the maximal elements of $\Int(\Psi)$ have rank $k\le r$ in $\Int(\Psi)$. It follows that all maximal elements of ${\cal M}(\Psi)$ have degree $k$ which is the requirement for purity.
\end{proof}

One particular feature of this conjecture is that it can be reduced to a statement about finitely many objects in a fixed rank. 
\begin{thm}\label{thm:reduction} Assume that the map ${\cal F}$ of Conjecture \ref{conj:main} exists for the class of all  rank $r$ ordered complexes QI$\cap$QE$\cap$FBP with $|E|\le 2r-1$. Then Conjecture \ref{conj:main} holds for rank $d$ complexes. 
\end{thm}

\begin{proof} Assume that Conjecture~\ref{conj:main} holds for all complexes $\Psi$ with $|E(\Psi)|\le 2r-1$ and let $\Psi$ be any ordered complex of rank $r$ satisfying QE$\cap$QI$\cap$FBP. For each basis $B$ disjoint from $B_0$ let $\mathbf{x}_B:= \prod_{b\in B}x_b$ and consider the set of monomials \begin{equation}{\cal F}(\Psi) : =\{x_B | B\in \B,\, B\cap B_0 = \emptyset\} \cup \left(\bigcup_{I\in \Delta|_{E-B_0}, \,|I|\le d-1} {\cal F}(\Psi|_{B_0\cup I})\right).\end{equation}

The multicomplexes in the union are defined since they correspond to complexes whose groundset has at most $2d-1$ elements, so they exist by assumption. Notice that there is one monomial for each basis $B$ of $B\in \B(\Psi)$. If $B_0 \cap B \not= \emptyset$, then $B$ can be found in any restriction $\Psi|_{B_0\cup I}$ for any $I$ containing $B\backslash B_0$ and $m_B$ comes from ${\cal F}(\Psi|_{B_0\cup I})$. The choice of $I$ is irrelevant: condition \emph{iv.} of the conjecture is satisfied, hence $m_B$ is the same monomial for all such complexes.

Condition $i.$ is straightforward to verify. Restrictions preserve internally passive sets, and so conditions $ii.$ and $iv.$ follow. To prove condition $iii.$ define the \emph{weak Gale order} $<_{wG}$ of $\Psi$ to be defined by $B_1<_{wG} B_2$ if there is an independent $I$ such that $B_1 < B_2$ in the Gale order of $\Psi|_{B_0\cup I}$. To check property $iii.$ it is enough to see that $wGale(\Psi) := (\B, <_{wG})$ is  between $\Int(\Psi)$ and  $\Gale(\Psi)$. 
\end{proof}

It can be checked that Conjecture~\ref{conj:main} refines Conjecture~3.11 in \cite{MR3339029}. The main difference is that the one in \cite{MR3339029} is formulated for very special kinds of ordered matroids. The restrictions on the orders can be replaced with the FBP. Consequently, the results of \cite{MR3339029} show that Conjecture~\ref{conj:main} holds for matroids of rank at most 4 by providing an algorithm that constructs ${\cal F}(\Psi)$. The proof is computer aided. Even better, a careful and straightforward analysis of the algorithm shows that the conjecture holds also for Gale truncations of matroids of rank at most 4. 

Internally perfect matroids defined by \cite{Dall} satisfy the conjecture and are a potential candidate for a big class of matroids satisfying Conjecture \ref{conj:main}. In particular, $\Int(\Psi)$ is the poset of divisibility of a pure multicomplex for internally perfect matroids, i.e., for internally perfect matroids the equality ${\cal M}(\Psi) = \Int(\Psi)$ holds. 

Considerations about Gale trunctations open the doors to do induction on the number of bases of the complex: one can construct the multicomplex for order ideals of Gale and there is at most one monomial that is not achieved by this technique. A good setting to guess how to do this type of construction is the class of pure shifted complexes; recall that these complexes belong to QI$\cap$QI$\cap$FBP and are closed under Gale truncations.

\begin{thm}\label{thm:shifted} Conjecture~\ref{conj:main} holds for shifted complexes.
\end{thm} 
\begin{proof} Recall that a rank-$d$ complex $\Psi = ([n], \Delta)$ is shifted if and only if $\Gale(\Psi)$ is an order ideal of Young's lattice of partitions contained in a $d\times (n-d)$ box. A basis of $\Psi$ with vertices $v_d>v_{d-1} > \dots >v_1$ corresponds to the partition $\lambda(B)  = \lambda_d\ge \lambda_{d-1} \ge \dots \ge \lambda_1$, where $\lambda_i = v_i -i$ (note that $\lambda_i$ might be equal to zero for some values of $i$). 

The goal is to construct a monomial $m_\lambda$ for each partition $\lambda:= \lambda(B)$. The following two properties are crucial in that process. 
\begin{compactitem}
\item $IA(B) = [m]$, where $m$ is the largest number such that $[m]\subseteq B$. This means that $\deg m_\lambda = |IP(B)| = \ell(\lambda)$, where $\ell(\lambda)$ denotes the length of $\lambda$. 
\item The set of variables of $m_\lambda$ are indexed by the set $B\backslash [d]$. In the partition such variables correspond to rows that intersect the main diagonal of the $d\times (n-d)$ box. It follows that the size of the support of $m_\lambda$ is equal to the side length of the Durfee square of $\lambda$, that is the maximal side length of a square formed by boxes that fits inside the Young diagram of the partition.  
\end{compactitem}

Notice that in small cases the construction is straightforward (see Figure~1). Let $\text{Dur}(\lambda)$ denote the partition obtained from $\lambda$ by removing the rows of the Durfee square. Assume that the side length of the Durfee square is $k$ and let $i_1<\dots <i_k$ be vertices of $B$ corresponding to the first $k$ rows. Now $\text{Dur}(\lambda)$ is a partition fitting in a $(d-k)\times k$ box. Hence by induction it has a monomial $x_{d-k+1}^{\alpha_{1}}x_{d-k+2}^{\alpha_2}\dots x_d^{\alpha_k}$ of degree $\ell(\lambda)-k$ (here some $\alpha_i$ may be zero). Let $m_\lambda := x_{i_1}^{\alpha_1+1}x_{i_2}^{\alpha_2 + 1} \dots x_{i_k}^{\alpha_k + 1}$. 

Let ${\cal F}(\Psi) =\{m_{\lambda(B)}\, | \,  $B$ \text{ is a basis of } \Psi \}$. We claim that ${\cal F}(\Psi)$ is the desired multicomplex. To show that this is a multicomplex, it is necessary to show that all the divisors of $m_\lambda$ are in ${\cal F}(\Psi)$ if $m_\lambda$ is. It suffices to show that all the divisors $m$ of $m_\lambda$ of degree one less than the degree of $m_\lambda$ are in ${\cal F}(\Psi)$. If the supports of $m$ and $m_\lambda$ are equal, then induction works directly passing to $\text{Dur}(\lambda)$. Otherwise let $\tilde m$ be the monomial resulting from dividing either $m$ or $m_\lambda$ by each one of its variables once (both choices give the same monomial). Let $e+k-1$ denote the degree of $m$, where $k-1$ is the number of variables of $m$. Our goal is to find $\mu$ such that $m = m_{\mu}$. The Durfee square of $\mu$ must have side length $k-1$. Let $j_1, \dots, j_{k-1}$ be the variables of $m$ and let $\tilde m = x_{j_1}^{\alpha_1} \dots x_{j_{k-1}}^{\alpha_{k-1}}$, with $\alpha_i \ge 0$ and $\sum \alpha_i = e$. By induction the monomial $x_{e+1}^{\alpha_1}\dots x_{e+k-1}^{\alpha_{k-1}}$ comes from a diagram $\tilde\mu \subseteq e\times (k-1)$. Thus we can construct $\mu$ by putting the first $k-1$ rows of $\mu$ in order to get the right support and putting $\tilde{\mu}$ below. It is straightforward to check that $\tilde \mu \subseteq \text{Dur}(\lambda)$: notice that, as diagrams contained in  the $e\times k$ box, the monomial associated to $\mu$ is $\tilde m$ and the monomial associated to $\text{Dur}(\lambda)$ is $x_{e+1}^{\alpha_1}\dots x_{e+j-1}^{\alpha_{j-1}}x_{e+j+1}^{\alpha_j} \dots x_{e+k}^{\alpha_{k+1}}$ for some $1\le j \le k-1$, thus $\text{Dur}(\mu)$ is obtained from $\tilde \mu$ by adding a box to the first $k-1-j$ rows. In this construction, it is also straightforward that $IP(B_{\mu}) \subseteq IP(B_{\lambda})$ by direct computation and the previous observation. 

Thus conditions $i.$, $ii.$ and $iii.$ of Conjecture 7.4 are satisfied. Condition $iv.$ is straighforward. 

There is a way to visualize $m_\lambda$ in the combinatorial structure of the Young diagram. The construction will be called the \emph{bouncing light construction}. Put mirrors on the vertical boundaries of the Young diagram of $\lambda$. The left-hand side mirrors reflect lines parallel to the $x$-axis in the direction of the diagonal and the right-hand side mirrors reflect lines coming in the direction of the diagonal to lines parallel to the $x$-axis. Put a light on the right-hand side of each row of the Durfee square and shoot the light parallel to the $x$-axis. For each $i_j$ let $\beta_j$ be the number of times that the light bounces off the left wall. Then a straightforward induction shows that $m_\lambda = x_{i_1}^{\beta_1}x_{i_2}^{\beta_2} \dots x_{i_k}^{\beta_k}$. The only thing we need to prove is that mirrors of all boxes are reached. This can be done by induction on $\ell(\lambda)$ by passing from $\lambda$ to $\text{Dur}(\lambda)$. 
\end{proof}

We end this section with two examples that illustrate the constructions. 
\begin{figure}\label{fig:young}
\begin{center}
\includegraphics[scale = 0.5]{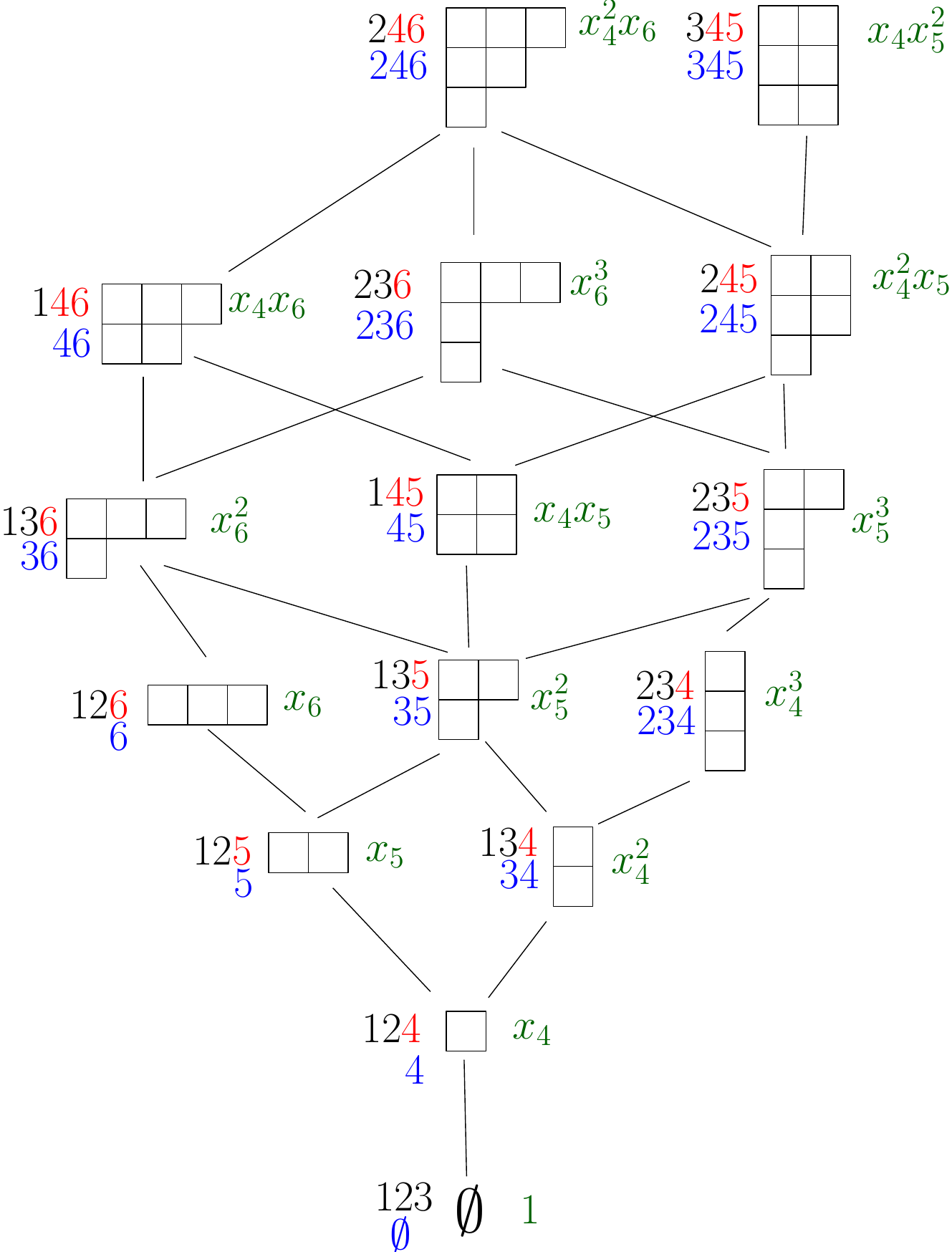}
\caption{An order ideal of the Young lattice}
\label{fig:Young}
\end{center}
\end{figure}
\begin{ex} Consider the order ideal of the Young lattice fitting into a $3\times 3$ box presented in Figure \ref{fig:Young}. Depicted in the figure are the Young diagrams. To the left of the diagram are listed the vertices of the corresponding facet with the variable vertices highlighted in red. Below the facets are the internally passive sets written in blue. To the right of the diagram is the corresponding monomial written in green. 
\end{ex}
\begin{figure}[htbp]
\begin{center}
\includegraphics[scale = 0.5]{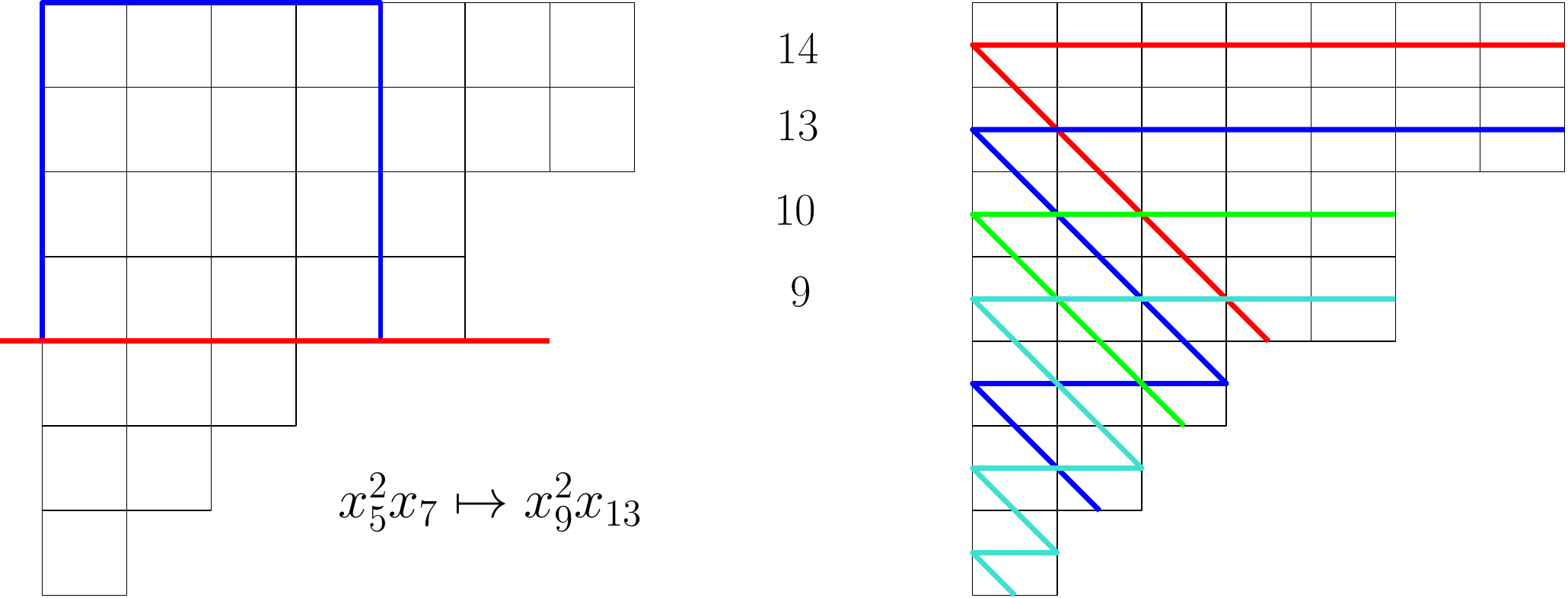}
\caption{Two ways to construct: inductively on the left and the bouncing light construction on the right.}
\label{fig:Young}
\end{center}
\end{figure}
\begin{ex} Consider the partition $(7,7,5,5,3,2,1)$. It fits into the $7\times 7$ box. It corresponds to the facet $\{14, 13, 10, 9,7,6,4\}$ in any complex it belongs to and the monomial associated is $x_{14}x^2_{13}x_{10}x^3_{9}$. The first four rows give the variables.\begin{compactitem} \item The left hand side of the figure shows the inductive construction. The blue square is the Durfee square and the top four rows give the answer. The remaining partition $(3,2,1)$ should be thought as fitting into the $3\times 4$ box. Hence it has a monomial in a subset of the variables $x_5, \, x_6, \, x_7, \, x_8$ and then we do an ordered substitution of variables: $x_5\mapsto x_9, \, x_6\mapsto x_{10}, \, x_7\mapsto x_{13}, \, x_8 \mapsto x_{14}$. 
\item The right hand side gives the construction of the bouncing light. The variables $x_{14},\,x_{13}, x_{10}, x_{9}$ correspond to colors {\color{red} red}, {\color{blue} blue}, {\color{green} green} and {\color{cyan} light blue} respectively. 
\end{compactitem}
\end{ex}

\section{Questions, remarks and future directions}
\subsection{Questions}
\begin{compactitem}
\item It would be interesting to find rich families of examples of complexes that are in one of QI, QE or QC that cannot be constructed from matroids and shifted complexes using standard operations. As shown in Theorem~\ref{thm:ex} such examples exist, but are typically found by ad hoc methods and, at present, there are no constructions of infinite families of examples. 
\item Notice that Theorem~\ref{thm:reduction} is slightly stronger than the analogue theorem proved in by Klee and the author in \cite{MR3339029}. In particular, the old proof required complexes with $2d$ elements in order to solve rank-$d$. This new reduction opens the door to the case of rank-$5$ matroids since it suffices to construct $\cal F$ for all rank-$5$ matroids with $9$ elements and those are fully classified in \cite{Mayhew-Royle}. Is there an algorithm similar to that by Klee and Samper that solves Conjecture~\ref{conj:main} for rank 5 matroids? 
\item Is it possible to use similar ideas to study oriented matroids. What is the correct definition of an oriented quasi-matroidal class of complexes? 
\end{compactitem}
\subsection{Remarks}
\begin{compactitem}
\item The class QE is related to the class of squarefree weakly polymatroidal ideals of Hibi and Kokubo \cite{MR2260118} that has been widely studied in commutative algebra. In particular, the Stanley Reisner ring of a complex in QE is weakly polymatroidal (after choosing the right conventions for the order). Mohammadi and Moradi \cite{MR2768496} showed that weakly polymatroidal ideals have linear quotients, which is an algebraic analogue for shellability. In fact, complexes that are weakly polymatroidal satisfy a similar axiom to the QE axiom, except that the element $b_1$ to be removed from a basis has to be the largest in the symmetric difference. Based on empirical information the flexibility of allowing various exchanges helps with constructions. However, we have not been able to find an example of a weakly polymatroidal complex that is not in QE. 

\item Other relaxations of matroid theory have been considered in the literature. In particular, Lenz \cite{MR3479578} considered collections of bases that satisfy what he calls the forward exchange axiom.

\end{compactitem}

{\footnotesize
\setstretch{0.95}
\bibliography{BIBLIO}
\bibliographystyle{myamsalpha}}
\end{document}